 \documentclass[11pt,a4paper,oneside,reqno]{amsart}
\usepackage{amsmath}
\usepackage{amssymb}
\usepackage{amstext}
\usepackage{amsfonts}
\usepackage{amsthm}
\usepackage{ifthen}
\usepackage{xspace}
\usepackage{comment}
\usepackage{graphicx}
\usepackage{mathrsfs}
\usepackage{enumerate}
\usepackage{breakcites}
\usepackage{xcolor}

\usepackage{hyperref}

\numberwithin{equation}{section}  

\newtheorem{punkt}{}[section]

\theoremstyle{plain}
\newtheorem{corollary}[punkt]{Corollary}
\newtheorem{lemma}[punkt]{Lemma}
\newtheorem{proposition}[punkt]{Proposition}
\newtheorem{theorem}[punkt]{Theorem}

\newcommand{\eins}{\leavevmode\hbox{\small1\kern-3.8pt\normalsize1}}
\theoremstyle{definition}
\newtheorem{remark}[punkt]{Remark}

\newtheorem{examples}[punkt]{Examples}

\newtheorem{definition}[punkt]{Definition}

\theoremstyle{plain}
\newtheorem*{corollary*}{Corollary}
\newtheorem*{lemma*}{Lemma}
\newtheorem*{proposition*}{Proposition}
\newtheorem*{theorem*}{Theorem}

\theoremstyle{definition}
\newtheorem*{remark*}{Remark}
\newtheorem*{remarks*}{Remarks}
\newtheorem*{example*}{Example}
\newtheorem*{examples*}{Examples}
\newtheorem*{problem*}{Problem}
\newtheorem*{problems*}{Problems}
\newtheorem*{question*}{Question}
\newtheorem*{questions*}{Questions}
\newtheorem*{definition*}{Definition}
\newtheorem*{conjecture*}{Conjecture}
\newtheorem*{assumption*}{Assumption}
\newtheorem*{assumptions*}{Assumptions}
\newtheorem*{construction*}{Construction}

\def\mycmplx{\mathbb{C}}

\def\mynat{\mathbb{N}}

\def\myreal{\mathbb{R}}

\def\mys{\mathcal{S}}

\def\ee{\mathbb{E}}		

\def\re{\qopname\relax{no}{Re}\,}
\def\im{\qopname\relax{no}{Im}\,}

\def\diag{\qopname\relax{no}{diag}}

\def\myparagraph#1{\paragraph{\textbf{#1.}}}

\def\eg{e.g.\@\xspace}
\def\ie{i.e.\@\xspace}

\def\GL{\qopname\relax{no}{GL}}
\def\SL{\qopname\relax{no}{SL}}
\def\Pos{\qopname\relax{no}{Pos}}

\def\laplace{\mathcal{L}}
\def\mellin{\mathcal{M}}

\def\FSV{f_{\operatorname{SV}}}		
\def\FEV{f_{\operatorname{EV}}}

\def\PF{\operatorname{PF}}			

\begin{document}

\title[Products of Random Matrices from Polynomial Ensembles]{Products of Random Matrices from \\ Polynomial Ensembles}
\author[Mario Kieburg and Holger K\"osters]{Mario Kieburg$^{1,2,*}$ and Holger K\"osters$^{3,\dagger}$}
 \address{$^1$Faculty of Physics \\
 University of Duisburg-Essen \\
 Lotharstr. 1, D-47048 Duisburg,
 Germany}
 \address{$^2$Department of Physics\\
 Bielefeld University\\
 Postfach 100131,
 D-33501 Bielefeld, Germany}
 \address{$^3$Department of Mathematics\\
 Bielefeld University\\
 Postfach 100131,
 D-33501 Bielefeld, Germany}
\email{$^*$ mkieburg@physik.uni-bielefeld.de}
\email{$^\dagger$ hkoesters@math.uni-bielefeld.de}

\begin{abstract}
Very recently we have shown that the spherical transform is a con\-venient tool 
for studying the relation between the joint density of the singular values 
and that of the eigenvalues for bi-unitarily invariant random matrices.
In the present work we discuss the implications of these results for products of random matrices.
In particular, we derive a transformation formula for the joint densities 
of a product of two independent bi-unitarily invariant random matrices,
the first from a polynomial ensemble
and the second from a polynomial ensemble of derivative type.
This allows us to re-derive and generalize a number of recent results in random matrix theory,
including a transformation formula for the kernels of the corresponding determinantal point processes.
Starting from these results, we construct a continuous family
of random matrix ensembles interpolating between the products of different numbers of
Ginibre matrices and inverse Ginibre matrices. Furthermore, we make contact to the asymptotic distribution of the Lyapunov exponents of the products of a large number of bi-unitarily invariant random matrices of fixed dimension.
\end{abstract}

\date{\today}

\maketitle


\section{Introduction}
\label{sec:Introduction}

The distributions of the singular values and the eigenvalues 
of products of independent complex random matrices
have been an intense subject of research in~the past few years.
Mostly, this was fuelled by two major developments. 
On the one hand, the asymptotic \emph{global} distributions
(also known as the macroscopic level densities) 
of the singular values and of the eigenvalues could be determined
for certain products of random matrices in the limit of large matrix dimension
\cite{ARRS:2013,AGT:2010,B:2011,BJW:2010,BJLNS:2011,BNS:2012,Forrester:2014,GKT:2015,GT:2010b,KT:2015,MNPZ:2014,OS:2011}, often with the help of free probability.
On the other hand, it was possible to investigate the \emph{local} correlations 
of the singular values and of the eigenvalues, both at finite and infinite matrix dimensions
\cite{ARRS:2013,AB:2012,AIK:2013,AKW:2013,CKW:2015,Forrester:2014,IK:2013,KS:2014,KZ:2014,LWZ:2014}.
See also the surveys \cite{Burda:2013} and \cite{AI:2015} for extended overviews.

For the investigation of the local spectral statistics,
the first step is usually the derivation 
of the joint probability density functions (henceforward called joint densities)
of the singular values or of the eigenvalues at finite matrix dimension.
As~witnessed by many results from the last years, the joint densities of matrix products
may still exhibit a determinantal structure when the joint densities of the underlying factors do. 
Such a structure is more than advantageous in studying the~spectral statistics and their asymptotics. 
For~instance, for a matrix product
\begin{align}
\label{eq:GaussianProduct}
X = Z_1 \cdots Z_p Z_{p+1}^{-1} \cdots Z_{p+q}^{-1} \,,
\end{align}
where $p,q \in \mynat_0$ are such that $p+q \geq 1$ 
and $Z_1,\hdots,Z_{p+q}$ are independent (complex) Ginibre matrices of dimension $n \times n$,
the joint density of the squared singular values is of the form
\begin{align}
\label{eq:SVD}
\FSV(a)=C_{\rm sv} \Delta_n(a) \det(w_{j-1}(a_k))_{j,k=1,\hdots,n} \,, \qquad a \in (0,\infty)^n
\end{align}
with a constant $C_{\rm sv}$, 
see e.g.\@ Refs.~\cite{AKW:2013,Forrester:2014,KZ:2014,LWZ:2014},
while that of the eigenvalues reads
\begin{align}
\label{eq:EVD}
\FEV(z)=C_{\rm ev} |\Delta_n(z)|^2 \prod_{j=1}^{n} w(|z_j|^2) \,, \qquad z \in \mathbb{C}^n
\end{align}
with another but related constant $C_{\rm ev}$, see e.g.\@ Refs.~\cite{AB:2012,ARRS:2013,IK:2013}. 
Here, $\Delta_n(x) := \prod_{1 \leq j < k \leq n} (x_k - x_j)$ is the Vandermonde determinant,
and $w_0,\hdots,w_{n-1}$ and $w$ are certain weight functions depending on $n$, $p$ and $q$, see Section~\ref{sec:InterpolatingEnsembles} for details.
Similar results were also established for products of rectangular matrices 
with independent complex Gaussian entries (induced Laguerre ensemble) 
\cite{AIK:2013,IK:2013,Forrester:2014} and of truncated unitary ensembles  
(induced Jacobi ensemble) \cite{IK:2013,ABKN:2014,KKS:2015,CKW:2015}. 
In all these cases, the functions $w_0,\ldots,w_{n-1}$ and $w$ admit compact expressions 
in terms of the Meijer-G function~\cite{Abramowitz}, which is why these ensembles 
of product matrices are also called \emph{Meijer G-ensembles} \cite{KK:2016a}.

The joint densities~\eqref{eq:SVD} and \eqref{eq:EVD} are closely related to each other. 
In the case of Meijer G-ensembles this relation is given by 
\begin{align}
\label{eq:link-0}
\Big( {-}a \partial_a \Big)^j w(a) = w_j(a) \,,\qquad j = 0,\hdots,n-1 \,.
\end{align}
In the very recent work~\cite{KK:2016a}, we constructed a linear operator 
mapping the joint density of the (squared) singular values to that of the eigenvalues,
see \cite[Section~3]{KK:2016a} for details.
This operator is called SEV (\underline{s}ingular value--\underline{e}igen\underline{v}alue) transform, 
and it~provides a bijection between the (squared) singular value and eigenvalue densities 
induced by general bi-unitarily invariant matrix densities, 
i.e.\@ by densities on matrix space which are unchanged when the argument is multiplied 
by an arbitrary unitary matrix from the left or from the right.
Furthermore, we were able to identify a subclass of polynomial ensembles
\cite{KS:2014,KKS:2015,Kuijlaars:2015,CKW:2015} called polynomial ensembles of derivative type
which admit joint densities of the form~\eqref{eq:SVD} and \eqref{eq:EVD},
with the weight functions related as in Eq.~\eqref{eq:link-0}.
This class extends the class of Meijer G-ensembles and even comprises some examples
of Muttalib-Borodin ensembles~\cite{Borodin,Muttalib}, which are generally no Meijer G-ensembles.

For some types of Meijer G-ensembles, as for induced Laguerre and Jacobi ensembles,
see e.g.\@ the review~\cite{AI:2015} and references therein, it is well known
that the operations of matrix multiplication and inversion do not lead 
out of the class of Meijer G-ensembles. 
It~is a natural question whether this statement extends to all polynomial ensembles of derivative type.
The main aim of the present work is to answer this question affirmatively.
Moreover, we will derive transformation formulas for the spectral densities
for the product of a random matrix from a polynomial ensemble of derivative type 
with an independent random matrix from an arbitrary polynomial ensemble, 
as discussed in \cite{KS:2014,KKS:2015,Kuijlaars:2015,CKW:2015}.

Let us briefly describe these results in more detail.
Let $X_1$ be an $n \times n$ random matrix drawn from a polynomial ensemble of derivative type,
\ie its matrix density is bi-unitarily invariant and the joint density
of its squared singular values is of the form \eqref{eq:SVD}, 
with weight functions $w_0,\hdots,w_{n-1}$ satisfying \eqref{eq:link-0}.
Moreover, let $X_2$ be an arbitrary $n \times n$ random matrix
for which the joint density of the squared singular values is of the form~\eqref{eq:SVD}, 
with weight functions $v_0,\ldots,v_{n-1}$. Then,
if $X_1$ and $X_2$ are independent, the joint density of the squared singular values 
of the product $X_1 X_2$ is also of the form \eqref{eq:SVD},
with weight functions $w \circledast v_0,\ldots,w \circledast v_{n-1}$,
where $\circledast$ denotes the multiplicative convolution on $(0,\infty)$, 
see Eq.~\eqref{eq:MellinConvolution} below.
Moreover, assuming additionally that the weight functions $v_j$ for the matrix $X_2$
satisfy Eq.~\eqref{eq:link-0} with $w$ replaced by $v$,
the weight functions $w \circledast v_j$ for the product $X_1 X_2$
satisfy Eq.~\eqref{eq:link-0} with $w$ replaced by $w \circledast v$.

From a technical perspective, our approach differs from the above-mentioned contributions to products of random matrices in that we use the (multivariate) \emph{spherical transform} instead of the (univariate) \emph{Mellin transform} as the main tool from harmonic analysis. See Section~\ref{sec:Preliminaries} for a short introduction. The spherical~trans\-form is a well-established tool for analysis and probability theory on matrix spaces \cite{FK,Helgason3,JL:SL2R,JL:SL2C,Terras}, and has been applied, for instance, to the study of the central limit theorem \cite{Sazonov-Tutubalin:1966,Bougerol:1981,Terras:1984,Richards:1989,GL:1995}, infinitely divisible distributions on matrix space \cite{Gangolli:1964,GL:1994,Applebaum:2001}, as well as problems in multivariate statistics \cite{HKKR:2011}.
However, it seems that this tool has not been exploited yet for the non-asymptotic investigation of the correlations of the singular values and of the eigenvalues of products of independent random matrices, a topic which has found considerable attention in the field of random matrix theory in the last few years.

For the derivation of our results, it will be essential that the \emph{spherical functions} associated with the group $\GL(n,\mycmplx)$ have the explicit representation \eqref{eq:SFDefinition}. This representation, which was discovered by Gelfand and Na\u{\i}mark~\cite{GelNai}, will serve as a substitute for the famous \emph{Harish-Chandra-Itzykson-Zuber integral} and related integrals. These integrals play a key role in the derivation of the squared singular value density \eqref{eq:SVD} for products of independent random matrices from the induced Laguerre and Jacobi ensemble.

As a first application of our approach, we will discuss the question whether it is possible to embed the matrix ensembles given by the products of Ginibre matrices and their inverses, see Eq.~\eqref{eq:GaussianProduct}, into a ``natural'' continuous family of matrix ensembles which is indexed by two positive parameters $p$~and~$q$ and which is closed under multiplicative convolution on matrix space; see Eq.~\eqref{eq:SphericalConvolution} below. We call the resulting ensembles \emph{interpolating ensembles}. In Section \ref{sec:InterpolatingEnsembles}, we will construct the interpolating densities and show that the corresponding joint densities of the squared singular values and eigenvalues are still of the form \eqref{eq:SVD} and \eqref{eq:EVD}, respectively, for appropriate choices of the weight functions $w_0,\hdots,w_{n-1}$ and $w$. Unfortunately, it turns out that the inter\-polating matrix densities are positive only under certain restrictions on the parameters $n$, $p$ and $q$. The same holds for the interpolating squared singular value densities, whereas the inter\-polating eigenvalue densities are always positive.

In addition to that, we will make contact with results about 
Lyapunov exponents and stability exponents for products of independent random matrices.
The investigation of limit theorems for Lyapunov exponents has a long history,
see~\eg \cite{Sazonov-Tutubalin:1966}, \cite{Bougerol-Lacroix:1985} and the references therein.
Interest in this area has resurged more recently due to explicit results for finite products
\cite{ABK:2014,Forrester:2013,Forrester:2015,Ipsen:2015,Kargin:2014,Reddy:2016}.
In particular, in~\cite{Reddy:2016} it~is shown that, under certain conditions, for products of independently and identically distributed, bi-unitarily invariant random matrices of fixed dimension, the logarithms of the singular values and  the complex eigenvalues are asymptotically Gaussian distributed as the number of factors tends to infinity.
We provide a~sketch of an alternative proof which is based on the spherical transform and which is reminiscent of the standard Fourier analytical proof of the central limit theorem for sums of random vectors. Our main motivation for including this sketch is that it provides some new insights about the class of polynomial ensembles of derivative type. For instance, it becomes clear that the Lyapunov exponents are asymptotically independent when the underlying random matrices are from this class. Moreover, we will obtain a more probabilistic characterization of the class of polynomial ensembles of derivative type.

The present work is organized as follows.
In Section \ref{sec:Preliminaries}, we introduce our notation and~recall the definition 
and the basic properties of the Mellin transform and the~spherical transform.
Also, we introduce the class of polynomial ensembles of derivative type
and summarize some results from Ref.~\cite{KK:2016a}.
Section~\ref{sec:MainResults} is devoted to the statements and the proofs 
of our main results on products of random matrices from polynomial ensembles. 
For illustration, we also provide some examples.
Section \ref{sec:InterpolatingEnsembles} contains the construction
and the investigation of the interpolating ensembles,
and Section~\ref{sec:Large-M} deals with the recent results
about Lyapunov exponents. Finally, we will summarize our results 
in Section~\ref{sec:conclusion} and give an outlook to some open questions.


\section{Preliminaries}
\label{sec:Preliminaries}

In this section, we introduce our notation and recall a number of known results
from random matrix theory and harmonic analysis which will be used later. 

We mainly use the notation of our work~\cite{KK:2016a}. 
For example, we write $\Delta_n(x) := \prod_{1 \le j < k \le n} (x_k - x_j)$ 
for the \emph{Vandermonde determinant} of a vector $x = (x_1,\hdots,x_n)$ in~$\myreal^n$ or~$\mycmplx^n$,
which we often identify with the diagonal matrix $x={\rm diag}(x_1,\ldots,x_n)$. 
Additionally, for abbreviation, we introduce the constant
\begin{equation}\label{const-def}
C_n^* := \frac{\pi^{n(n+1)/2}}{ \prod_{j=0}^{n-1} j!} \,,
\end{equation}
which appears in several integration formulas. 


\subsection{Matrix Spaces, Densities \& Operators}
\label{subsec:MatrixSpaces}

Throughout the present work, we  denote the general linear group of complex $n \times n$ matrices by 
$G = \GL(n,\mycmplx)$, the group of  $n \times n$ unitary matrices by
$K = {\rm U}(n)$, and the group of complex (upper) unitriangular matrices by $T$.
We endow the matrix spaces $G$ and $T$ with the respective \emph{Lebesgue measures}
\begin{align}
\label{eq:LebesgueMeasures}
dg = \prod_{j,k=1,\hdots,n} dg_{jk} \quad \text{and} \quad dt= \prod_{1 \le j < k \le n} dt_{jk} \,,
\end{align}
where $g\in G$ and $t\in T$. Here the measure $dz$ denotes the Lebesgue measure on $\myreal$ if $z$ is a real variable
and the Lebesgue measure $dz = d\re z \, d\im z$ on $\mycmplx$ if $z$ is a complex variable. 
Let us emphasize that $dg$ does not denote integration with respect to the Haar measure on $G$,
which is given by
\begin{align}
\label{eq:HaarMeasures}
d^*g = \frac{dg}{|\det g|^{2n}} \,.
\end{align}
For the unitary group $K$, we always employ the normalized Haar measure on $K$
denoted by $d^*k$, so that $\int_K d^*k = 1$.
Finally, given a matrix $g \in \GL(n,\mycmplx)$, we~write $g^*$ for the Hermitian adjoint. 

By a \emph{density} on a matrix space, we understand a Borel measurable function
which is Lebesgue integrable with respect to the corresponding reference measure.
Note~that, unless otherwise indicated, we do not assume a density to be non-negative.
Sometimes, but not always, we write \emph{signed density} to emphasize this.
When it is important that a density is non-negative, we call it a \emph{non-negative density},
or a \emph{probability density} if it~is additionally normalized.
In contrast to~that, when we speak about \emph{random matrices} or about \emph{ensembles},
we always mean that the matrix space under consideration is endowed with a \emph{probability density}.
\emph{Independence} of random matrices always means statistical independence.

Given a density $f_G$ on $\GL(n,\mycmplx)$, we will frequently consider 
the induced joint densities of the squared singular values and of the eigenvalues,
which are defined on the sets $A := \myreal_+^n$ and $Z := \mycmplx_*^n$, respectively,
where $\myreal_+ := (0,\infty)$ and $\mycmplx_*=\mycmplx\setminus\{0\}$.
The arguments of these densities are denoted by $a = (a_1,\hdots,a_n) \in A$
and $z = (z_1,\hdots,z_n) \in Z$, respectively, 
and they are identified with diagonal $n \times n$ matrices when convenient. 
The reference measure on $A$~and~$Z$ will always be given by the Lebesgue measure 
on $\myreal^n$ restricted to~$A$ and on $\mycmplx^n$ restricted to~$Z$, respectively.
Finally, we always consider versions of the joint densities 
which are invariant with respect to permutations of their arguments,
so that the ordering of the squared singular values and the eigenvalues is irrelevant.

For some of our results, it will be useful to specify the sets of densities
and the linear operators describing the induced densities explicitly;
see also Ref.~\cite{KK:2016a} for more information.
The set of bi-unitarily invariant densities on the matrix space $G$ is denoted by
\begin{equation}
L^{1,K}(G):=\{f_G\in L^1(G) \,|\, f_G(k_1gk_2)=f_G(g)\ \forall\ k_1,k_2\in K\ {\rm and}\ g\in G\}. \label{def-L1G}
\end{equation}
On this space, we consider two operators $\mathcal{I}$ and $\mathcal{T}$ defined by
\begin{equation}
\mathcal{I}f_G(a):=\frac{(C_n^*)^2}{\pi^n \, n!}|\Delta_n(a)|^2f_G(\sqrt{a}) \label{I-def}
\end{equation}
and
\begin{equation}
\mathcal{T}f_G(z) := C_n^*|\Delta_n(z)|^2\left(\prod_{j=1}^n|z_j|^{2(n-j)}\right) \int_{T} f_G(z t) dt, \label{T-def}
\end{equation}
where the integration domain $T$ and  the measure $dt$ are as in Eq.~\eqref{eq:LebesgueMeasures}.
The images of these operators are denoted by 
\begin{align}
\label{imagespaces}
L^{1,{\rm SV}}(A):=\mathcal{I} L^{1,K}(G)
\quad\text{and}\quad
L^{1,{\rm EV}}(Z):=\mathcal{T}L^{1,K}(G),
\end{align}
respectively. The notation of these sets are due to the fact that if we start from a density $f_G \in L^{1,K}(G)$,
$\mathcal{I}f_G$ and $\mathcal{T}f_G$ are the induced joint densities of the squared singular values
and of the eigenvalues, respectively. For this reason, we frequently write
$\FSV$ instead of $\mathcal{I}f_G$ and $\FEV$ instead of $\mathcal{T}f_G$.
In Ref.~\cite{KK:2016a} it was shown that the operators $\mathcal{I}$ and $\mathcal{T}$ are invertible. 
For the reader familiar with \cite{KK:2016a}, let us mention that the operator $\mathcal{I}$ defined here
is the composition $\mathcal{I}_A \mathcal{I}_\Omega$ of the operators $\mathcal{I}_A$ and $\mathcal{I}_\Omega$
in \cite{KK:2016a}. In particular, the map $\mathcal{R} := \mathcal{T}\mathcal{I}^{-1}$
is a bijection called the \emph{SEV transform} in Ref.~\cite{KK:2016a}
and this bijection may be written quite explicitly.
Also, let us mention that the operators $\mathcal{I}$ and $\mathcal{T}$
map probability densities to probability densities.
The same property holds for the inverse operator $\mathcal{I}^{-1}$,
whereas it may fail for the inverse operator $\mathcal{T}^{-1}$;
see \cite[Section~3]{KK:2016a} or Section~4 below for details.


\subsection{Definition of some Transforms}
\label{subsec:Transforms}
Next, let us recall the definition and the basic properties of the \emph{Mellin transform} 
and of the \emph{spherical transform}. \pagebreak[2]

For a function $f \in L^1(\myreal_{+})$, the \emph{Mellin transform} is defined by
\begin{align}
\label{eq:MellinDefinition}
\mellin f(s) := \int_0^\infty f(x) \, x^{s-1} \, dx \,.
\end{align}
It is defined for all~those $s \in \mycmplx$ such that the integral exists (in the Lebesgue sense).
In~particular, if $f \in L^1(\myreal_+)$, the Mellin transform is defined at least on the line $1 + \imath \myreal$, 
and it has the following well-known properties:

\begin{enumerate}
\item[(i)]	\emph{Uniqueness theorem.} 
If $f_1$ and $f_2$ are in $L^1(\myreal_+)$ and their Mellin transforms coincide
on the set $1 + \imath \myreal$, we have $f_1 = f_2$ almost everywhere.
Indeed, there exist explicit \emph{Mellin inversion formulas} 
by which one can recover the original function from its Mellin transform,
see \eg \cite{Titchmarsh} or \cite{KK:2016a}.
 \item[(ii)]	\emph{Multiplication theorem.} If $f_1$ and $f_2$ are in $L^{1}(\myreal_+)$ and 
\begin{align}
\label{eq:MellinConvolution}
(f_1 \circledast f_2)(x) := \int_0^\infty f_1(xy^{-1}) f_2(y) \, \frac{dy}{y}
\end{align}
denotes their multiplicative convolution, then
\begin{align}
\label{eq:MellinMultiplication}
\mellin([f_1 \circledast f_2];s) = \mellin f_1(s) \, \mellin f_2(s)
\end{align}
for all $s \in \mathbb{C}$ such that the Mellin transforms on the right are defined.
This~multiplication theorem follows from a simple calculation. 
\item[(iii)]	\emph{Composition with derivatives.} 
For $f\in L^{1,k}_{\mathbb{I}}(\mathbb{R}_+)$, we have
\begin{align}
\label{eq:mellin-link}
\mellin \Big(\Big[ \Big({-}x \frac{d}{dx} \Big)^k f(x) \Big]; s \Big)= s^k \, \mellin f(s).
\end{align}
\end{enumerate}
In part (iii), we have used the notation $\mellin ([f(x)];s)$ instead of $\mellin f(s)$
in~order to indicate the argument of the underlying function.
The set $L^{1,k}_{\mathbb{I}}(\mathbb{R}_+)$ is~defined by
\begin{multline}
 L^{1,k}_{\mathbb{I}}(\mathbb{R}_+) := \bigg\{f\in L^{1}(\mathbb{R}_+) \,\Big|\, 
 \text{$f$ is $k$-times differentiable and} \\
 \text{for all $\kappa \in\mathbb{I}$ and $j=0,\ldots k$} :
 \int_0^\infty \bigg|y^{\kappa-1}\Big({-}y\frac{\partial}{\partial y}\Big)^jf(y)\bigg|dy<\infty\bigg\}
\label{func-def}
\end{multline}
with $\mathbb{I} \subset \mathbb{R}$ an interval containing the number $1$.
Here, ``$k$-times differentiable'' means $(k-1)$-times continuously differentiable
with an absolutely continuous $(k-1)$'st derivative. As is well known,
this implies that the $k$th derivative exists almost everywhere.
The set $L^{1,k}_{\mathbb{I}}(\mathbb{R}_+)$ will also play a role in the definition 
of the polynomial ensembles of derivative type, see Section~\ref{subsec:PolynomialEnsembles} below.

We now introduce the spherical transform. In doing so, we confine ourselves to a direct definition 
which is sufficient for our purposes and refer to the literature \cite{Helgason3,Terras,FK,JL:SL2R,JL:SL2C}
for background information.
Also, let us emphasize that we define the spherical transform 
for bi-unitarily invariant densities on the group $\GL(n,\mycmplx)$ 
in the present work,
whereas it was defined for unitarily invariant densities
on the~cone $\Pos(n,\mycmplx)$ of positive-definite Hermitian matrices
in Ref.~\cite[Section 2.4]{KK:2016a}. For a function $f_G \in L^{1,K}(G)$, the \emph{spherical transform} is
\begin{equation}
\label{S-def}
\mys f_G(s) := \int_G f_G(g) \, \varphi_s(g) \, \frac{dg}{|\det g|^{2n}},
\end{equation}
where
\begin{equation}
\label{eq:SFDefinition}
\varphi_s(g) = \frac{\Delta_n(\varrho')}{\Delta_n(s)} \, \frac{\det \big[(\lambda_j(g^* g))^{s_{k}+(n-1)/2} \big]_{j,k=1,\hdots,n}} {\Delta_n(\lambda(g^* g))} \,, \qquad s \in \mathbb{C}^n \,,\ g \in G \,,
\end{equation}
denotes the \emph{spherical function} \cite[Theorem IV.5.7]{Helgason3}
for the group $G$. In this~de\-finition, we introduced the vector
 \begin{equation}\label{rho-def}
 \varrho':=(\varrho'_1,\ldots,\varrho'_n) \quad {\rm with} \quad \varrho'_j=(2j+n-1)/2, \quad j=1,\ldots,n,
 \end{equation}
and $\lambda(g^* g) = (\lambda_j(g^* g))_{j=1,\hdots,n}$ is the vector of the eigenvalues of $g^* g$ or, equivalently, the squared singular values of $g$. The spherical functions $\varphi_s$ in Eq.~(2.14) 
are bi-unitarily invariant and satisfy the equation
\begin{equation}
\label{eq:SFCharacterization}
\int_K \varphi_s(g k h) \,  d^* k = \varphi_s(g) \, \varphi_s(h) \qquad \text{for all $g,h \in G$\,,}
\end{equation}
see \eg \cite[Prop. IV.2.2]{Helgason3}. The representation \eqref{eq:SFDefinition} goes back to 
Gelfand and Na\u{\i}mark \cite{GelNai}. The spherical transform $\mathcal{S} f_G(s)$ is defined 
for all those $s \in \mycmplx^n$ such~that the integral exists (in the Lebesgue sense). 
Note that, by the definitions, we have 
\begin{equation}
\label{eq:SFNormalization}
\mathcal{S} f_G(\varrho')=\int_G f_G(g) \, dg
\end{equation}
for $\varrho'$ as in Eq.~\eqref{rho-def}. In particular, $\mathcal{S} f_G(\varrho') = 1$ if $f_G$ is a probability density. Also, if $f \in L^{1,K}(G)$, the spherical transform is defined at least on the set $\varrho' + \imath \myreal^n$, and it satisfies the following well-known properties:
 
\begin{enumerate}
\item[(i)]		
\emph{Uniqueness theorem.} 
If $f_1$ and $f_2$ are in $L^{1,K}(G)$ 
and their spherical transforms coincide on the set $\varrho' + \imath \myreal^n$, 
we have $f_1 = f_2$ almost everywhere,
see e.g.\@ Ref.~\cite[Chapters IV.3 and IV.8]{Helgason3}. 
Furthermore, there exist explicit \emph{spherical inversion formulas}
by which one can recover a function $f \in L^{1,K}(G)$ from its spherical transform,
see e.g.\@ \cite[Lemma 2.9]{KK:2016a} for a version which does not require any smoothness conditions.
\item[(ii)]	\emph{Multiplication theorem.}		
Let $f_{G,1}$ and $f_{G,2}$ be functions in $L^{1,K}(G)$,
and let their multiplicative convolution be defined by
\begin{align}
\label{eq:SphericalConvolution}
(f_{G,1} \circledast f_{G,2})(g)
:= 
\int_G f_{G,1}(g \tilde{g}^{-1}) f_{G,2}(\tilde{g}) \frac{d\tilde{g}}{|\det \tilde{g}|^{2n}} \,.
\end{align}
It is straightforward to see that this is well-defined for almost all $g \in G$
and that $f_{G,1} \circledast f_{G,2}$ is again 
an element of $L^{1,K}(G)$. Then we have
\begin{align}
\label{eq:SphericalMultiplicationNew}
\mathcal{S}(f_{G,1} \circledast f_{G,2})(s) = \mathcal{S}f_{G,1}(s) \mathcal{S}f_{G,2}(s)
\end{align}
for all $s \in \mathbb{C}^n$ such that the spherical transforms on the right are defined.
This follows from a simple calculation using Eq.~\eqref{eq:SFCharacterization} and bi-unitary in\-variance, 
compare \eg \cite[Theorem IV.6.3]{JL:SL2R}.

Equation~\eqref{eq:SphericalMultiplicationNew} has the following probabilistic reformulation,
which will be convenient in the next section.
If $X_1$ and $X_2$ are independent random matrices
with densities $f_{G,1}$ and $f_{G,2}$ in $L^{1,K}(G)$, respectively,
$f_{G_1} \circledast f_{G_2}$ is simply the density of the matrix product $X_1 X_2$.
Thus, writing $\mys_X$ instead of $\mys f_G$ for a random matrix $X$
with a density $f_G \in L^{1,K}(G)$, we may rewrite Eq.~\eqref{eq:SphericalMultiplicationNew}
in the form
\begin{align}
\label{eq:SphericalMultiplicationX}
\mathcal{S}_{X_1 X_2}(s) = \mathcal{S}_{X_1}(s) \mathcal{S}_{X_2}(s) \,.
\end{align}
We also call $\mys_X$ the spherical transform of the random matrix $X$.
\end{enumerate}


\subsection{Polynomial Ensembles}
\label{subsec:PolynomialEnsembles}

\emph{Polynomial ensembles} \cite{KS:2014,KKS:2015,Kuijlaars:2015} 
play a prominent role in the investigation of singular value statistics,
as their subclass of the \emph{poly\-nomial ensembles of derivative type} \cite{KK:2016a}. 
Therefore, let us recall their definitions.

\begin{definition}[Polynomial Ensembles] \ 
\label{def:PE}

\noindent{}Fix $n \in \mynat$, and let $w_0,\hdots,w_{n-1}\in L^{1,0}_{[1,n]}(\mathbb{R}_+)$ and $\omega\in L^{1,n-1}_{[1,n]}(\mathbb{R}_+)$.

\begin{enumerate}[(a)]
\item
A probability density $\FSV^{(n)}[w]\in L^{1,{\rm SV}}(A)$ defines a \emph{polynomial ensemble} 
if it is of the form
\begin{align}
\label{eq:PE-density}
\FSV^{(n)}([w];a)=C_{\rm sv}^{(n)}[w] \, \Delta_n(a) \, \det[w_{j-1}(a_k)]_{j,k=1,\hdots,n}
\end{align}
with the normalizing constant
\begin{align}
\label{eq:PE-normalization}
C_{\rm sv}^{(n)}[w] := \left( n!\det\bigg[\int_0^\infty a^{k-1}w_{j-1}(a)da\bigg]_{j,k=1,\hdots,n} \right)^{-1} \,.
\end{align}
It corresponds to a \emph{polynomial random matrix ensemble} $f_G^{(n)}[w]=\mathcal{I}^{-1}\FSV^{(n)}[w]$ on $G$. These ensembles are also called the poly\-nomial ensembles associated with the weight functions $\{w_j\}_{j=0,\ldots,n-1}$.
\item
Suppose \emph{additionally} that 
\begin{align}
\label{eq:link}
w_k(a)=\left({-} a \frac{d}{da}\right)^k\omega(a) \,, \qquad k = 0,\hdots,n-1,
\end{align}
i.e.\@ we have
\begin{align}
\label{eq:PED-density}
\FSV^{(n)}([\omega];a)=C_{\rm sv}^{(n)}[\omega] \, \Delta_n(a) \, \det\left[\left({-} a_k \frac{d}{da_k}\right)^{j-1} \omega(a_k)\right]_{j,k=1,\hdots,n}
\end{align}
with the normalizing constant
\begin{align}
\label{eq:PED-normalization}
C_{\rm sv}^{(n)}[\omega] := \frac{1}{\prod_{j=0}^{n}j!}\,\frac{1}{\prod_{j=1}^{n}\mathcal{M}\omega(j)} \,.
\end{align}
Then we say that the probability density given by Eq.~\eqref{eq:PED-density} 
defines a \emph{poly\-nomial ensemble of derivative type}. 
The associated density $f_G^{(n)}[\omega]=\mathcal{I}^{-1}\FSV^{(n)}[\omega]$ 
on $G$ is called a \emph{polynomial random matrix ensemble of derivative type}.
These ensembles are also called the poly\-nomial ensembles (of derivative type)
asso\-ciated with the weight function $\omega$. 
For brevity, we often omit the attribute ``of~derivative~type'' here,
as this is unlikely to cause misunderstandings.
\end{enumerate}
\end{definition}

Let us note that the integrability conditions in the Definition \ref{def:PE} 
ensure that the functions $w_0,\ldots,w_{n-1}$ and $\omega$ as well as 
the densities \eqref{eq:PE-density} and \eqref{eq:PED-density} are integrable. 
Furthermore, the Mellin transforms $\mellin w_k$, $k=0,\hdots,n-1$, and $\mathcal{M}\omega$
exist (at least) in the complex strip $[1,n] + \imath \myreal$. 
Also, it is part of the definition that the functions on the right hand side 
of Eqs. \eqref{eq:PE-density} and \eqref{eq:PED-density} are non-negative and normalizable.
It can be checked that the normalizing constants are indeed given 
by Eqs. \eqref{eq:PE-normalization} and \eqref{eq:PED-normalization}, respectively,
see \cite{KK:2016a} or Section~3 below.
Moreover, let us emphasize that polynomial random matrix ensembles
are bi-unitarily invariant by definition.

Occasionally, similarly as in \cite{KK:2016a}, we will need the extension of the previous definitions
to signed densities. By slight abuse of notation, we call the resulting signed measures on $A$ and on $G$
\emph{signed polynomial ensembles}.

Moreover, similarly as in the definition above, we often make explicit 
the underlying weight functions $w = \{ w_j \}_{j=0,\hdots,n-1}$ or $\omega$ in square brackets
when speci\-fying normalization constants, spectral densities, correlation kernels and functions,
\linebreak[2] bi-orthogonal kernels and functions etc.

Finally, let us quote some results from Ref.~\cite{KK:2016a} for polynomial random matrix ensembles of derivative type. 
The first result provides the joint density~\eqref{eq:PED-density} of the eigenvalues.

\begin{theorem}[{\cite[Theorem 3.5]{KK:2016a}}]\
\label{thm:EigenvalueDistribution}

Let $X$ be a random matrix drawn from the polynomial random matrix ensemble 
of derivative type associated with the weight function $\omega\in L^{1,n-1}_{[1,n]}(\mathbb{R}_+)$.
Then the joint density of the eigenvalues of $X$ is given by
\begin{align}
\label{eq:EigenvalueDistribution-11}
\FEV^{(n)}([\omega];z) =\mathcal{R}\FSV^{(n)}([\omega];a)=C_{\rm ev}^{(n)}[\omega] \, |\Delta_n(z)|^2 \, \prod_{j=1}^{n} \omega(|z_j|^2) \,,
\end{align}
where
\begin{align}
\label{eq:CEV}
C_{\rm ev}^{(n)}[\omega] := \frac{C_{\rm sv}^{(n)}[\omega] \prod_{j=0}^{n-1} j! }{\pi^n } \,
\end{align}
and $C_{\rm sv}^{(n)}[\omega]$ is as in Eq.~\eqref{eq:PED-normalization}.
\end{theorem}

The second result shows that the joint densities of the squared singular values and of the eigenvalues
give rise to determinantal point processes.

\begin{lemma}[{\cite[Lemmas 4.1 \& 4.2]{KK:2016a}}]\
\label{lem:Kernels}

Let $X$ be a random matrix as in Theorem~\ref{thm:EigenvalueDistribution} with the additional condition that $\omega\in L^{1,n}_{]s_{\min},s_{\max}[}(\mathbb{R}_+)$ and $[1,n+1]\subset {]s_{\min},s_{\max}[}\subset\mathbb{R}$. Then the joint densities of the squared singular values and eigenvalues of $X$ give rise to the determinantal point processes
\begin{equation}\label{deter-sv}
 \FSV^{(n)}([\omega];a)=\frac{1}{n!}\det\Big[K_{\rm sv}^{(n)}([\omega];a_b,a_c)\Big]_{b,c=1,\ldots,n},
\end{equation}
and
\begin{equation}\label{deter-ev}
 \FEV^{(n)}([\omega];z)=\frac{1}{n!}\det\Big[K_{\rm ev}^{(n)}([\omega];z_b,\bar{z}_c)\Big]_{b,c=1,\ldots,n},
\end{equation}
respectively.
The kernels are given by
\begin{eqnarray}\label{ker-sv-a}
K_{\rm sv}^{(n)}([\omega];a_b,a_c)&=&\sum_{j=0}^{n-1}p_j([\omega];a_b)q_j([\omega];a_c)\\
&=&-n\frac{\mathcal{M}\omega(n+1)}{\mathcal{M}\omega(n)}\int_0^1p_{n-1}([\omega];xa_b)q_n([\omega];xa_c)dx \nonumber
\end{eqnarray}
and
\begin{equation}\label{ker-ev}
K_{\rm ev}^{(n)}([\omega];z_b,\bar{z}_c)=\sqrt{\omega(|z_b|^2)\omega(|z_c|^2)}\sum_{j=0}^{n-1}\frac{(z_b\bar{z}_c)^j}{\pi\mathcal{M}\omega(j+1)} \,.
\end{equation}
For the kernel of the squared singular values of $X$ we have the polynomials in monic normalization
\begin{eqnarray}\label{pol-sv}
p_{l}([\omega];a)&=&\sum_{j=0}^l(-1)^{l-j}\frac{l!\mathcal{M}\omega(l+1)}{j!(l-j)!\mathcal{M}\omega(j+1)}a^j\\
&=&l! \mathcal{M}\omega(l+1)\oint_{\mathcal{C}}\frac{\Gamma(t-l-1)}{\Gamma(t)\mathcal{M}\omega(t)}a^{t-1}\frac{dt}{2\pi\imath },\nonumber
\end{eqnarray}
 $l=0,\ldots,n-1$. The closed contour $ \mathcal{C}$ encircles the interval $[1,n]$ and satisfies ${\rm Re}\, \mathcal{C}\subset {]s_{\min},s_{\max}[}$. Moreover it excludes any poles of $1/\mathcal{M}\omega(t+1)$. The functions
\begin{eqnarray}\label{func-sv}
q_{l}([\omega];a)
&=&\frac{1}{l!\mathcal{M}\omega(l+1)}\partial_a^l\left[(-a)^l\omega(a)\right]\\
&=&\frac{1}{l!\mathcal{M}\omega(l+1)}\lim_{\epsilon\to0}\int_{-\infty}^\infty\frac{\pi^2\cos(\epsilon s)}{\pi^2-4\epsilon^2 s^2}\frac{\Gamma(s_0+\imath  s)\mathcal{M}\omega(s_0+\imath  s)}{\Gamma(s_0+\imath  s-l)}
a^{-s_0-\imath  s}\frac{ds}{2\pi}\nonumber
\end{eqnarray}
 are bi-orthogonal to the polynomials~\eqref{pol-sv} such that $\int_0^\infty p_{l}([\omega];a)q_{m}([\omega];a)$ $da=\delta_{lm}$ for $l,m=0,\ldots,n$ and $\delta_{lm}$ the Kronecker symbol. The auxiliary real shift $s_0$ is chosen such that $s_0\in {]s_{\min},1[}$ and $s_0<{\rm Re}\, t$ for all $t\in\mathcal{C}$.
\end{lemma}

After these preparations, we are ready to formulate our main results.


\section{Main Results}
\label{sec:MainResults}

To shorten the notation, let us convene that in the proofs in this section, matrices inside determinants 
are indexed by $j,k = 1,\hdots,n$. Also, recall the notation $\mys_X$ introduced 
above Eq.~\eqref{eq:SphericalMultiplicationX}.

The derivation of the joint densities of the eigenvalues and the singular values 
of a product of independent random matrices from polynomial ensembles 
works via the spherical transform \eqref{S-def} 
and the multiplication and uniqueness theorems from Section \ref{subsec:Transforms}.
To this end, we first need to know the spherical transform of a~poly\-nomial ensemble.

\begin{proposition}[Spherical Transform of a Polynomial Ensemble]\ 
\label{prop:SphericalTransform}

Let $X$ be a random matrix from the polynomial random matrix ensemble 
associated with the weight functions $\{w_j\}_{j=0,\ldots,n-1}$.
Then the spherical transform of $X$ is given by
\begin{align}
\label{eq:SphericalTransform-11}
\mys_X (s) = C_{\rm sv}^{(n)}[w]\Big( \prod_{j=0}^{n} j! \Big) \frac{\det[\mellin w_{j-1}(s_k-(n-1)/2)]_{j,k=1,\hdots,n}}{\Delta_n(s)} \,.
\end{align}
\end{proposition}

Note that for $s = \varrho'$ as in \eqref{rho-def}, we have $\mys_X(\varrho') = 1$ and $\Delta_n(\varrho') = \prod_{j=0}^{n-1} j!$, so~that we recover Eq.~\eqref{eq:PE-normalization} for the normalizing constant $C_{\rm sv}^{(n)}[w]$.

\begin{proof}
Let $\FSV$ be the density of the squared singular values of $X$, see Eq.~\eqref{eq:PE-density}, 
and let $f_G := \mathcal{I}^{-1} \FSV$ with $\mathcal{I}$ as in Eq.~\eqref{I-def}
be the corresponding matrix density.
Then, using Eqs. \eqref{S-def}, \eqref{eq:PE-density} and \eqref{eq:SFDefinition},
we have 
\begin{multline}
\mys_X(s) 
= \int_G f_G(g) \varphi_s(g) \, \frac{dg}{|\det g|^{2n}} 
= \int_A \FSV(a) \, \varphi_s(\sqrt{a}) \, \frac{da}{(\det a)^n} \\
=C_{\rm sv}^{(n)}[w] \left(\prod_{j=0}^{n-1}j!\right)\int_A\Delta_n(a) \, 
	\det[w_{j-1}(a_k)]\frac{\det[a_k^{s_j-(n-1)/2}]}{\Delta_n(s)\Delta_n(a)}\frac{da}{\det a}\,.
\label{eq:spherical-transform-2}
\end{multline}
The Vandermonde determinant $\Delta_n(a)$ cancels with the one in the numerator and the factor $1/\Delta_n(s)$ can be pulled out 
of the integral. For the remaining integral, we apply Andr\'eief's identity~\cite{Andreief} and end up with
\begin{equation}
\mys_X(s) =\frac{C_{\rm sv}^{(n)}[w]\prod_{j=0}^{n} j! }{\Delta_n(s)} 
	\det\left[\int_0^\infty a^{s_k-(n-1)/2} \, w_{j-1}(a) \, \frac{da}{a}\right]
\end{equation}
The integral in the determinant is the Mellin transform~\eqref{eq:MellinDefinition}, which completes the~proof.
\end{proof}

\begin{corollary}[Spherical Transform of a Polynomial Ensemble \mbox{of Derivative Type}]
\label{cor:SphericalTransform}

Let $X$ be a random matrix drawn from the polynomial random matrix ensemble
associated with the weight function $\omega\in L^{1,n-1}_{[1,n]}(\mathbb{R}_+)$. Then
\begin{align}
\label{eq:SphericalTransform-12}
\mys_X (s) = \prod_{k=1}^{n} \frac{\mellin \omega(s_k-(n-1)/2)}{\mellin \omega(k)} \,.
\end{align}
\end{corollary}

\begin{proof}
By Assumption \eqref{eq:link} and Eq.~\eqref{eq:mellin-link}, we have 
$\mellin w_{j-1}(s) = s^{j-1}\,\mellin \omega(s)$, $j=1,\hdots,n$.
It therefore follows from Proposition~\ref{prop:SphericalTransform}
via elementary column trans\-formations that
\begin{align}
\label{eq:spherical-transform-6}
   \mys_X (s) 
&= C_{\rm sv}^{(n)}[\omega] \, \Big( \prod_{j=0}^{n} j! \Big) \frac{\det((s_k-(n-1)/2)^{j-1})_{j,k=1,\hdots,n}}{\Delta_n(s)}
	\prod_{k=1}^{n} \mellin \omega\left(s_k-\frac{n-1}{2}\right) \,.
\end{align}
Here the fraction cancels out by the translation-invariance
$\Delta_n(x_1+x,\hdots,x_n+x) = \Delta_n(x_1,\hdots,x_n)$ of the Vandermonde determinant.
After that, setting $s = \varrho'$ and using that $\mys_X(\varrho') = 1$,
we see that the normalization constant $C_{\rm sv}^{(n)}[\omega]$ is given by
Eq.~\eqref{eq:PED-normalization}.
Inserting this into Eq.~\eqref{eq:spherical-transform-6} completes the proof.
\end{proof}

With the help of Proposition~\ref{prop:SphericalTransform} and Corollary~\ref{cor:SphericalTransform},
we can now readily derive the following \emph{transfer formula} for the joint density of the singular values 
when a random matrix from a polynomial ensemble is multiplied by an independent random matrix 
from a polynomial ensemble of derivative type. This generalizes recent results by Kuijlaars et al.
\cite{KS:2014,KKS:2015,Kuijlaars:2015,CKW:2015}, where only products with induced Laguerre (chiral Gaussian) en\-sembles
or induced Jacobi (truncated unitary) en\-sembles were considered.

\begin{theorem}[Transfer for Polynomial Ensembles]\
\label{thm:Transfer}

Let $X_1$ and $X_2$ be independent random matrices
from polynomial random matrix ensembles associated with the weight functions
$\omega\in L^{1,n-1}_{[1,n]}(\mathbb{R}_+)$ and $w := \{w_j\}_{j=0,\ldots,n-1}$, respectively. 
Then the matrix product $X_1 X_2$ belongs to the polynomial random matrix ensemble 
associated with the functions $\omega \circledast w := \{\omega \circledast w_j\}_{j=0,\ldots,n-1}$, 
where we employ the multiplicative convolution~\eqref{eq:MellinConvolution}. 
In particular, the joint density of the squared singular values of $X_1 X_2$
is given by
\begin{align}
\label{trans-sv}
\FSV^{(n)}([\omega \circledast w];a)=C_{\rm sv}^{(n)}[\omega \circledast w] \, \Delta_n(a) \, \det[(\omega \circledast w_{j-1})(a_k)]_{j,k=1,\hdots,n}\,.
\end{align}
\end{theorem}

\begin{proof}
By Proposition \ref{prop:SphericalTransform}, Corollary \ref{cor:SphericalTransform}
and the multiplication theorem~\eqref{eq:SphericalMultiplicationX} for the spherical transform, 
we have
\begin{align}
\mys_{X_1 X_2}(s) &= \mys_{X_1}(s) \, \mys_{X_2}(s) \nonumber\\
&= \frac{C_{\rm sv}^{(n)}[w]\prod_{j=0}^{n} j!}{\prod_{k=1}^{n}\mellin \omega(k)} 
	\frac{\det[\mellin w_{j-1}(s_k-(n-1)/2)]}{\Delta_n(s)} \prod_{k=1}^{n} \mellin \omega\left(s_k - \frac{n-1}{2}\right) \,.
\end{align}
Here we may absorb the product over $\mellin \omega(s_k - (n-1)/2)$ into the determinant. Similarly, the product in the denominator can be absorbed 
into the normalization constant $C_{\rm sv}^{(n)}[w]$, see its definition~\eqref{eq:PE-normalization}. 
Using the multiplication theorem~\eqref{eq:MellinMultiplication} for the Mellin transform, it then follows that
\begin{align}
\mys_{X_1X_2}(s)=C_{\rm sv}^{(n)}[\omega\circledast w]\Big(\prod_{j=0}^{n} j!\Big) 
	\frac{\det[ \mellin (\omega\circledast w_{j-1})(s_k-(n-1)/2) ]}{\Delta_n(s)}.
\end{align}
By construction, this is the spherical transform of a probability density on $G$,
and by the uniqueness theorem for the spherical transform
and Proposition \ref{prop:SphericalTransform},
this density can only be that of the polynomial ensemble
associated with the weight functions $\{\omega \circledast w_j\}_{j=0,\ldots,n-1}$. 
\end{proof}

Theorem~\ref{thm:Transfer} simplifies drastically when the second matrix is also from a poly\-nomial ensemble of derivative type. 
In particular, in this case, we also obtain a transfer formula for the joint density of the eigenvalues under multiplication.

\begin{corollary}[Transfer for Polynomial Ensembles of Derivative Type]\
\label{cor:Transfer}

Let $X_1$ and $X_2$ be independent random matrices
from polynomial random matrix ensembles associated with the weight functions
$\omega_1\in L^{1,n-1}_{[1,n]}(\mathbb{R}_+)$ and $\omega_2\in L^{1,n-1}_{[1,n]}(\mathbb{R}_+)$, respectively. 
Then the matrix product $X_1 X_2$ belongs to the polynomial random matrix ensemble 
associated with the function $\omega_1\circledast \omega_2 \in L^{1,n-1}_{[1,n]}(\mathbb{R}_+)$.
In~particular,
\begin{enumerate}[(a)]
\item
the joint density of the squared singular values of $X_1 X_2$ is given by
\begin{multline}
\label{eq:SingularValueDistribution-12}
\FSV^{(n)}([\omega_1\circledast \omega_2];a)=C_{\rm sv}^{(n)}[\omega_1\circledast \omega_2]\Delta_n(a) \\
\,\times\, \det\left[\left({-} a_k \frac{d}{da_k}\right)^{j-1} (\omega_1\circledast \omega_2)(a_k)\right]_{j,k=1,\hdots,n},
\end{multline}
\item
the joint density of the eigenvalues of $X_1 X_2$ is given by
\begin{align}
\label{eq:EigenvalueDistribution-12}
\FEV^{(n)}([\omega_1\circledast \omega_2];z) = \frac{C_{\rm sv}^{(n)}[\omega_1\circledast \omega_2] \prod_{j=0}^{n-1} j! }{\pi^n } \, |\Delta_n(z)|^2 \, \prod_{j=1}^{n} (\omega_1\circledast \omega_2)(|z_j|^2).
\end{align}
\end{enumerate}
\end{corollary}

\begin{proof}
It is straightforward to check that $\omega_1\circledast \omega_2 \in L^{1,n-1}_{[1,n]}(\mathbb{R}_+)$
and that \linebreak $({-}x (d/dx))^k (\omega_1\circledast \omega_2)(x)$ 
coincides with the multiplicative convolution of the func\-tions 
$\omega_1(x)$ and $({-}x (d/dx))^k\omega_2(x)$, $k = 0,\hdots,n-1$.
It therefore follows from Theorem \ref{thm:Transfer} that $X_1 X_2$ belongs 
to the polynomial matrix ensemble associated with the function $(\omega_1\circledast \omega_2)(x)$.
Thus, part (a) holds by definition, while part (b) follows from Theorem \ref{thm:EigenvalueDistribution}.
\end{proof}

Let us recall that the multiplicative convolution on $L^{1,K}(G)$ is commutative 
due to the bi-unitary invariance.
This consequence of the bi-unitary invariance has even a counterpart for rectangular matrices 
and is reflected in the weak commutation relation shown in Ref.~\cite{IK:2013}. 
This permutation symmetry was also observed for products of Meijer G-ensembles, 
see the reviews~\cite{AB:2012,Burda:2013}.
In the particular case of polynomial ensembles of derivative type, 
this commutative behavior also follows from the preceding result and the commutativity
of the multiplicative convolution on $L^1(\myreal_{+})$.

\begin{remark}[Spherical Transform of an Inverse Random Matrix]\
\label{rem:InverseMatrix}

Let $X$ be a bi-unitarily invariant random matrix on $G$ with the probability density $f_G(g)$.
Then the inverse random matrix $X^{-1}$ is also a bi-unitarily invariant random matrix on $G$, 
with the probability density $f_G(g^{-1}) |\det g|^{-4n}$,
as follows by the change of variables $g \to g^{-1}$ because $dg^{-1}=dg/|\det g|^{4n}$.

In the special case where $X$ is a random matrix drawn from the polynomial matrix ensemble
associated with the weight function $\omega$, 
it can be shown that
\begin{align}
\mys_{X^{-1}}(s) = \mys_X(2n-s) = \prod_{j=1}^{n} \frac{\mellin\omega\left(1+n-s_j+(n-1)/2)\right)}{\mathcal{M}\omega(1+n-j)} ,
\end{align}
where we used the relation $\varphi_s(x^{-1}) = \varphi_{-s}(x)$,
see \eg \cite[Eq.~(IV.4.7)]{Helgason3}, in the first step
and Corollary~\ref{cor:SphericalTransform} in the second step.
Thus, arguing similarly as in the proof of Theorem \ref{thm:Transfer},
we may conclude that $X^{-1}$ is a random matrix from the polynomial matrix ensemble
associated with the~weight~function $\widetilde{\omega}(a) := \omega(a^{-1}) a^{-n-1}$.
\end{remark}

Clearly, using the preceding results, it is fairly easy to derive the distributions
of products of independent random matrices from polynomial matrix ensembles
of derivative type and/or their inverses. Furthermore, 
similarly as in Corollary~\ref{cor:Transfer},
we also obtain the joint densities of the squared singular values 
and of the eigenvalues for such products.


Let us give a few examples of polynomial matrix ensembles of derivative type
to which these results may be applied, see also \cite[Example 3.4]{KK:2016a}
and the references therein:

\begin{examples} \ 
\label{ex:transforms}
\begin{enumerate}[(a)]
\item 
For $\nu>-1$, let $X$ be a random matrix from the polynomial random matrix ensemble 
associated with the weight function $\omega_{\rm Lag}(a) = a^{\nu} \exp(-a)$. 
This is known \cite{KK:2016a} to be the \emph{induced Laguerre ensemble}.
In particular, when $\nu = 0$, this reduces to the Ginibre ensemble.
The spherical transform of $X$ is explicitly given by
\begin{align}
\mys_X(s) = \prod_{j=1}^{n} \frac{\Gamma(\nu+s_j-(n-1)/2)}{\Gamma(\nu+j)} \,,
\end{align}
which is essentially the multivariate Gamma function, see \eg \cite[Thm.~VII.1.1]{FK}.
The spherical transform of the inverse random matrix $X^{-1}$ can be derived 
by Remark~\ref{rem:InverseMatrix} and reads
\begin{align}
\mys_{X^{-1}}(s) =\prod_{j=1}^{n} \frac{\Gamma(\nu-s_j+(3n+1)/2)}{\Gamma(\nu+n+1-j)} \,.
\end{align}
\item
For $\mu> n-1$ and $\nu >-1$, let $X$ be a random matrix from the polynomial random matrix ensemble 
associated with the weight function $\omega_{\rm Jac}(a) = a^{\nu} (1-a)^{\mu-1}\pmb{1}_{(0,1)}(y)$,
where $\pmb{1}_{(0,1)}$ denotes the indicator function of the interval $(0,1)$.
This is known \cite{KK:2016a} to be the \emph{induced Jacobi ensemble},
which can be identified with an ensemble of truncated unitary matrices
when $\mu$ and $\nu$ are integers. The spherical transform of $X$ is given by
\begin{align}
\mys_X(s) =\prod_{j=1}^{n} \frac{\Gamma(\nu+\mu+j)\Gamma(\nu+s_j-(n-1)/2) }{\Gamma(\nu+j)\Gamma(\nu+\mu+s_j-(n-1)/2)} \,,
\end{align}
which is related to the multivariate Beta function, see \eg \cite[Thm.~VII.1.7]{FK},
and that of the inverse random matrix $X^{-1}$ is
\begin{align}
\mys_{X^{-1}}(s) =\prod_{j=1}^{n} \frac{\Gamma(\nu+\mu+n+1-j)\Gamma(\nu-s_j+(3n+1)/2) }{\Gamma(\nu+n+1-j)\Gamma(\nu+\mu-s_j+(3n+1)/2)}.
\end{align}
\item 
For $\mu>n-1$ and $\nu>-1$, let $X$ be a random matrix from the polynomial random matrix ensemble 
associated with the weight function $\omega_{\rm CL}(x) = x^{\nu} (1+x)^{-\mu-\nu-1}$.
This ensemble is also known as the \emph{Cauchy-Lorentz ensemble}. 
Then, the spherical transform of $X$ is given by
\begin{align}
\mys_X(s) =\prod_{j=1}^{n} \frac{\Gamma(\nu+s_j-(n-1)/2) \, \Gamma(\mu-s_j+(n+1)/2)}{\Gamma(\nu+j) \, \Gamma(\mu+1-j)} \,,
\end{align}
and that of $X^{-1}$ is
\begin{align}
\mys_{X^{-1}}(s) =\prod_{j=1}^{n} \frac{\Gamma(\nu-s_j+(3n+1)/2) \, \Gamma(\mu+s_j-(3n-1)/2)}{\Gamma(\nu+n+1-j) \, \Gamma(\mu-n+j)} \,.
\end{align}
\item The first three ensembles were more or less classical random matrix ensembles. In particular they are all Meijer G-ensembles, which is reflected by the fact that their spherical transforms are products of Gamma functions. In Ref.~\cite{KK:2016a} we have shown that the Muttalib-Borodin ensemble~\cite{Muttalib,Borodin} is in general not a Meijer G-ensemble (as is also visible from Eq.~\eqref{MB1} below), but it is still a polynomial random matrix ensemble of derivative type. Note that we consider the bi-unitarily invariant matrix version of this ensemble, and not the one proposed in Ref.~\cite{Forrester-Wang}. 

Let $X$ be a random matrix drawn from the bi-unitarily invariant 
Muttalib-Borodin ensemble of Laguerre-type. This is the polynomial random matrix ensemble 
associated with the weight function $\omega_{\rm MB}(a)=a^\nu e^{-\alpha a^\theta}$
with $\nu>-1$ and $\alpha,\theta>0$. This weight generates a second Vandermonde determinant 
in the joint density of the squared singular values, namely $\Delta_n(a^\theta)$.  
The spherical transform of $X$ is
\begin{align}
\label{MB1}
\mys_X(s) =\prod_{j=1}^{n} \frac{\alpha^{(2j-2s_j+n-1)/(2\theta)}\,\Gamma((2\nu+2s_j-n+1)/(2\theta))}{\Gamma((\nu+j)/\theta)} 
\end{align}
and that of $X^{-1}$ is
\begin{align}
\mys_{X^{-1}}(s) =\prod_{j=1}^{n} \frac{\alpha^{(2s_j-2j-n+1)/(2\theta)}\,\Gamma((2\nu-2s_j+3n+1)/(2\theta))}{\Gamma((\nu+n+1-j)/\theta)} \,.
\end{align}
\item 
In the limit as $\theta \to 0$, the Muttalib-Borodin ensemble can be approximated
by another ensemble which is not a Meijer G-ensemble.
The weight function becomes $\omega_{\theta\to0}(a)=a^{\nu'} e^{-\alpha'({\rm ln}\, a)^2}$
with $\alpha'=\theta^2\alpha/2>0$ and $\nu'=\nu-\alpha\theta\in\mathbb{R}$ fixed.
This weight generates a second Vandermonde determinant $\Delta_n({\rm ln}\, a)$
in the joint density of the squared singular values.
Here, the spherical transforms of $X$ and $X^{-1}$ turn out to be
\begin{align}
\mys_X(s) =\prod_{j=1}^{n} \exp\left(\frac{(\nu'+s_j-(n-1)/2)^2-(\nu'+j)^2}{4\alpha'}\right)
\end{align}
and
\begin{align}
\mys_{X^{-1}}(s) =\prod_{j=1}^{n} \exp\left(\frac{(\nu'-s_j+(3n+1)/2)^2-(\nu'+n+1-j)^2}{4\alpha'}\right),
\end{align}
respectively, \ie they are of \emph{Gaussian} form. As a consequence,
the family of these polynomial random matrix ensembles is ``stable'' under multiplicative convolution. 
More precisely, when $X_1$ and $X_2$ are independent random matrices drawn 
from the polynomial ensembles associated with weight functions $\omega_{1}(a)=a^{\nu_1} e^{-\alpha_1({\rm ln}\, a)^2}$ and $\omega_{2}(a)=a^{\nu_2} e^{-\alpha_2({\rm ln}\, a)^2}$, respectively, the product $X_1 X_2$ 
belongs to the polynomial random matrix ensemble associated with the weight function 
$\omega_{12}(a)=a^{\nu_{12}} e^{-\alpha_{12}({\rm ln}\, a)^2}$, 
where $1/\alpha_{12}=1/\alpha_1+1/\alpha_2$ and $\nu_{12}/\alpha_{12}=\nu_1/\alpha_1+\nu_2/\alpha_2$. 
Hence this distribution plays a similar role for random matrices 
as the log-normal distribution for scalar random variables, 
to which it reduces for $n = 1$. In fact, it is also known 
as the \emph{Gaussian measure} on~$\GL(n,\mycmplx)$, see~\eg~\cite{GL:1995}.
Also, this distribution is \emph{infinitely divisible} 
with respect to multiplicative convolution on $\GL(n,\mycmplx)$.
\end{enumerate}
\end{examples}

All of these ensembles give rise to relatively simple joint densities 
for the squared singular values and eigenvalues of their products, 
where the weight function $\omega$ is either a Meijer G-function~\cite{Abramowitz} 
or some generalization of this function. For example, let us consider 
the product $X_1X_2X_3$ where $X_1$ is drawn from an induced Laguerre ensemble, 
$X_2$ is drawn from a Muttalib-Borodin ensemble of the type when $\theta\to0$, 
and $X_3$ is the inverse of an induced Jacobi random matrix, 
and all matrices are independent. Then the product $X_1X_2X_3$ 
belongs to the polynomial random matrix ensemble associated with the weight function
\begin{equation}
\omega_{X_1X_2X_3}(a) \propto \frac{1}{2\pi\imath} \int_{-\imath\infty}^{+\imath\infty} \frac{\Gamma(\nu_1+s)\Gamma(\nu_3-s+n+1)}{\Gamma(\nu_3+\mu_3-s+n+1)}e^{(\nu_2+s)^2/(4\alpha_2)}a^{-s} \, ds \,.
\end{equation}
Note that some of these ensembles like the Laguerre and Jacobi ensembles were already studied in the literature, see the review~\cite{AI:2015} and references therein. 
But some of them were not studied yet, possibly because the group integrals 
involved seemed very complicated. 
With the help of our approach, all these products can now be treated in a unified and simple way.

In particular, when multiplying independent random matrices from polynomial ensembles of derivative type, the changes of eigenvalue and singular value statistics in terms of the kernels 
and their bi-orthogonal polynomials become very simple due to Lemma~\ref{lem:Kernels}. 
However, what happens with the statistics when we multiply a random matrix 
from a polynomial random matrix ensemble of derivative type by an independent 
random matrix from a general polynomial random matrix ensemble? All we have to do is to combine 
Theorem~\ref{thm:Transfer} with Ref.~\cite[Lemma 2.14]{CKW:2015} by Claeys et al.
This leads to the following corollary.

\begin{corollary}[Transformation of the Kernel for Squared Singular Values]\
\label{cor:Kernel-trafo}

Let $X_1,X_2$ be the random matrices considered in Theorem~\ref{thm:Transfer}. Suppose that the bi-orthogonal system corresponding to $X_2$ consists of the polynomials (in monic normalization)
\begin{equation}\label{polyn-def}
 p_k([w];x)=\sum_{j=0}^k a_{jk}[w]x^j,\ {\rm with}\ k=0,\ldots,n-1,\ a_{jk}[w]\in\mathbb{R}\ {\rm and}\ a_{kk}[w]=1
\end{equation}
and the weight functions $\{q_k([w])\}_{k=0,\ldots, n-1}\subset{\rm span}(w_0,\ldots,w_{n-1})$ bi-orthogonal to the polynomials~\eqref{polyn-def}. 
Also, suppose that this bi-orthogonal system satisfies the normalization $\int_0^\infty p_k([w];a)q_k([w];a)da=1$ for all $k=0,\ldots,n-1$,
so that the kernel of the determinantal point process of the squared singular values of $X_2$ has the form
\begin{equation}\label{kernel-def}
 K_{\rm sv}^{(n)}([w];a_1,a_2)=\sum_{k=0}^{n-1} p_k([w];a_1)q_k([w];a_2).
\end{equation}
Moreover we define the function
\begin{equation}\label{aux-func-def}
 \chi^{(n)}([\omega];x)=\sum_{j=0}^{n-1}\frac{x^j}{\mathcal{M}\omega(j+1)}
\end{equation}
Then, the joint density of the squared singular values of the product $X_1 X_2$ gives rise to a determinantal point process, too. The corresponding bi-orthogonal system consists of the polynomials (in monic normalization)
\begin{eqnarray}\label{polyn-def-b}
 p_k([\omega\circledast w];x)&=&\sum_{j=0}^k a_{jk}[\omega\circledast w]x^j\\
 &=&\sum_{j=0}^k \frac{\mathcal{M}\omega(k+1) a_{jk}[w]}{\mathcal{M}\omega(j+1)}x^j\nonumber\\
 &=&\frac{\mathcal{M}\omega(k+1)}{2\pi\imath }\oint \chi^{(n)}([\omega];s)p_k\left([w];\frac{x}{s}\right)\frac{ds}{ s},\ k=0,\ldots,n-1,\nonumber
\end{eqnarray}
where the contour encircles the origin, and the weights
\begin{eqnarray}\label{weight-def}
 q_k([\omega\circledast w]; x)&=&\frac{(\omega\circledast q_k[w])(x)}{\mathcal{M}\omega(k+1)}\\
 &=&\frac{1}{\mathcal{M}\omega(k+1)}\int_0^\infty\omega(t) q_k\left([w];\frac{x}{t}\right)\frac{dt}{t},\ k=0,\ldots,n-1.\nonumber
\end{eqnarray}
This bi-orthogonal system satisfies the normalization 
\mbox{$\int_0^\infty p_k([\omega\circledast w];a)q_k([\omega\circledast w];a)\, da$} $=1$,
so that the corresponding kernel has the form
\begin{equation}\label{kernel-def-b}
 K_{\rm sv}^{(n)}([\omega\circledast w];a_1,a_2)=\frac{1}{2\pi\imath }\int_0^\infty\omega(t)\left(\oint \chi^{(n)}([\omega];s)K_{\rm sv}^{(n)}\left([w];\frac{a_1}{s},\frac{a_2}{t}\right) \frac{ds}{ s}\right)\frac{dt}{t}.
\end{equation}
\end{corollary}

Note that the function $ \chi^{(n)}([\omega];x)$ is essentially the kernel $K_{\rm ev}^{(n)}([\omega];z_b,\bar{z}_c)$, see Eq.~\eqref{ker-ev}, at $z_b=z_c=\sqrt{x}$ divided by $\omega(x)$. However, for this identification we need the assumption that $\omega(x)>0$ for $x>0$.

Indeed these transformation identities immediately reduce to the replacement $\omega\to \omega_1\circledast \omega_2$ for the quantities in Lemma~\ref{lem:Kernels} when considering a product of two independent random matrices $X_1$ and $X_2$ from polynomial ensembles of derivative type, associated with the weight functions $\omega_1$ and $\omega_2$. 
Even more, in this special case, the kernel of the eigenvalues for the product $X_1 X_2$ is given by
\begin{multline}
K_{\rm ev}^{(n)}([\omega_1\circledast \omega_2];z_b,\bar{z}_c)\\
=\sqrt{(\omega_1\circledast \omega_2)(|z_b|^2)(\omega_1\circledast \omega_2)(|z_c|^2)}\sum_{j=0}^{n-1}\frac{(z_b\bar{z}_c)^j}{\pi\mathcal{M}\omega_1(j+1)\mathcal{M}\omega_2(j+1)}. \label{ker-ev-trafo}
\end{multline}
It is not clear whether there exists an analogue of such a formula when $X_2$ is drawn from a general polynomial random matrix ensemble. The reason is that an explicit expression, where all integrals are performed exactly, for the joint density $\FEV^{(n)}[w]=\mathcal{T}f_G^{(n)}[w]=\mathcal{R}\FSV^{(n)}[w]$ is not known in the general case.

\begin{proof}[Proof of Corollary \ref{cor:Kernel-trafo}]
Most parts of the corollary follow from the proof of \cite[Lemma 2.14]{CKW:2015}, since the functions involved here satisfy the properties required like integrability and positivity of $\omega$. The only difference to \cite[Lemma 2.14]{CKW:2015} is the function $ \chi^{(n)}([\omega];x)$, but it is clear that this may be used in the same way as the Laurent series in \cite[Lemma 2.14]{CKW:2015}. It is even simpler to prove the corollary with the finite sum~\eqref{aux-func-def} because we do need not to care about the convergence of the series.
\end{proof}

As a final remark, let us point out that in all our results concerning products of random matrices,
one can replace one of the bi-unitarily invariant random matrices by a positive diagonal random matrix,
or in fact any other random matrix, with the same squared singular value density. Nothing in the results
will change.


\section{Interpolating Densities}
\label{sec:InterpolatingEnsembles}

Let us return to the spectral densities of the product \eqref{eq:GaussianProduct}
of $p$ Ginibre matrices and $q$ inverse Ginibre matrices, all of them independent,
and to the question of whether it is possible to interpolate between these densities.

\medskip

\myparagraph{Exponential Distribution}
For later reference, we recall some basic facts about the exponential distribution,
\ie the probability distribution on $(0,\infty)$ with density $h(x) = e^{-x}$.
This distribution is additively \emph{and} multiplicatively infinitely divisible.
For the latter property, note that if $X$ has the density $h(x) = e^{-x}$ on~$\myreal_{+}$,
then $Y=-\ln X$ has the density $g(y) = \exp[-y-e^{-y}]$ on~$\myreal$.
This is the density of the \emph{Gumbel distribution},
which is known to be additively infinitely divisible,
see \eg \cite[Example 11.1]{Steutel-vanHarn}.
 
In the sequel, for any $q > 0$, we write $h_q(x)$ 
for the $q$th multiplicative convolution power of the exponential density
and, analogously, $g_q(y)$ for the $q$th additive convolution power of the Gumbel density.
By infinite divisibility, $h_q(x)$ and $g_q(y)$ are again 
prob\-ability densities.
Since $\mellin h(s) = \Gamma(s)$ and $\laplace g(s) = \Gamma(1+s)$,
where $\laplace$ denotes the \emph{Laplace transform},
we have the representations
\begin{align}
\label{eq:hq-representation}
h_q(x) = \frac{1}{2\pi\imath} \int_{c-\imath\infty}^{c+\imath\infty} \Gamma^q(s) \, x^{-s} \, ds \qquad (x \in (0,\infty), c > 0)
\end{align}
and
\begin{align}
\label{eq:gq-representation}
g_q(y) = \frac{1}{2\pi\imath} \int_{c-\imath\infty}^{c+\imath\infty} \Gamma^q(1+s) \, e^{ys} \, ds  \qquad (y \in \myreal, c > -1) \,.
\end{align}
To define fractional powers of $\Gamma(z)$, we always take the analytic branch of the~logarithm 
of $\Gamma(z)$ on $\mycmplx \,\setminus\, (-\infty,0]$ which is positive on $(0,\infty)$. 
The integrals in \eqref{eq:hq-representation} and \eqref{eq:gq-representation} 
are absolutely convergent for $q>0$, because, by Stirling's formula for $\Gamma(z)$, we~have
\begin{align}
|\Gamma(c+it)| = \mathcal{O}\left(|t|^{c-1/2} \, e^{-\pi |t|/2}\right)
\label{eq:stirling}
\end{align}
as $|t| \to \infty$, $t \in \myreal$, uniformly in $c$ for any compact interval $I \subset \myreal$.
Furthermore, it~follows from Eq.~\eqref{eq:stirling} that $h_q(x)$ and $g_q(y)$ are infinitely often differentiable. Finally, let us recall that the densities $g_q(y)$ and $h_q(x)$ satisfy the relation
\begin{align}
\label{eq:gh-relations}
g_q(y) = h_q(e^{-y}) e^{-y} \, \,.
\end{align}
Thus,
it~follows by induction that
\begin{align}
\label{eq:gh-derivatives}
\frac{d^k}{dy^k} g_q(y) = e^{-y}\left.\left(-x\frac{d}{dx}-1\right)^k h_q(x)\right|_{x=e^{-y}} 
\end{align}
for all $k \in \mynat_0$.

\medskip

\myparagraph{Interpolating Densities}
Consider the random matrix $X$ from Eq.~\eqref{eq:GaussianProduct}. 
It~follows from Example \ref{ex:transforms}\,(a), Corollary \ref{cor:Transfer}
and the basic properties of the Mellin transform that this matrix belongs to 
the polynomial matrix ensemble of derivative type associated with the weight function 
\begin{align}
\label{eq:weight-2}
w^{(p,q)}(x) :\!\!&= \frac{1}{2\pi\imath} \int_{c-\imath\infty}^{c+\imath\infty} \Gamma^p(s) \, \Gamma^q(1+n-s) \, x^{-s} \, ds \,,\quad  x > 0\ \text{and}\ c\in(0,n+1).
\end{align}
In particular, the joint density of the eigenvalues of $X$ is given~by
\begin{align}
\label{eq:Interpolating-EVD}
\FEV^{(p,q)}(z) =C_{\rm ev}^{(n)}[w^{(p,q)}] |\Delta_n(z)|^2 \prod_{j=1}^{n} w^{(p,q)}(|z_j|^2) \,, \qquad z \in \mathbb{C}^n \,,
\end{align}
and that of the squared singular values of $X$ reads
\begin{align}
\label{eq:Interpolating-SVD}
\FSV^{(p,q)}(a)=C_{\rm sv}^{(n)}[w^{(p,q)}] \Delta_n(a) \, \det(w^{(p,q)}_{j-1}(a_k))_{j,k=1,\hdots,n} \,, \qquad a \in (0,\infty)^n \,,
\end{align}
where $w_0^{(p,q)}(x),\hdots,w_{n-1}^{(p,q)}(x)$ are obtained from $w^{(p,q)}(x)$ as in Eq.~\eqref{eq:link-0}
and the normalizing constants $C_{\rm ev}^{(n)}[w^{(p,q)}]$ and $C_{\rm sv}^{(n)}[w^{(p,q)}]$
are given by Eqs.~\eqref{eq:CEV} and \eqref{eq:PED-normalization}.
These results were originally derived in \cite{AB:2012, ARRS:2013} 
and \cite{AKW:2013, Forrester:2014}, respectively.
For~the sake of clarity, let~us mention that the weight functions are
partly different in these references, but the resulting determinants are the same.

Thus far, $p$ and $q$ are non-negative integers with $p+q > 0$. By means of~inter\-polation, 
it seems natural to consider the function $w^{(p,q)}(x)$
for any non-negative reals $p$~and~$q$ such that $p+q > 0$. 
However, then the question arises
whether the corresponding densities \eqref{eq:Interpolating-EVD} and \eqref{eq:Interpolating-SVD}
are non-negative, and hence probability densities.
We will call these densities \emph{interpolating eigenvalue densities}
and \emph{interpolating squared singular value densities}, respectively.

It follows from the analyticity of $s \mapsto \Gamma^p(s) \, \Gamma^q(1+n-s)$ 
in the~strip \mbox{$(0,n+1) + \imath \myreal$} and the asymptotic relation \eqref{eq:stirling}
that the function $w^{(p,q)}(x)$ is well-defined and independent from the choice of $c \in (0,n+1)$.
Additionally, it~is clear that $w^{(p,q)}(x) \in L^{1,n-1}_{(0,n+1)}(\myreal_{+})$. In~particular, 
this implies that the functions \eqref{eq:Interpolating-EVD} and \eqref{eq:Interpolating-SVD}
are in fact integrable. Finally, employing the basic properties of the Mellin transform, we find that 
\begin{align}
\label{eq:weight-6}
w^{(p,q)}(x) =
\begin{cases}
h_p(x), & \text{for $p > 0,\,q = 0$,} \\
h_q(x^{-1}) \, x^{-(n+1)}, & \text{for $p = 0,\,q > 0$,} \\[+3pt]
\displaystyle\int_0^\infty h_p(xy^{-1}) \, h_q(y^{-1}) \, y^{-(n+1)} \, dy/y, & \text{for $p > 0,\,q > 0$.}
\end{cases}
\end{align}
In particular, this representation shows that the functions $w^{(p,q)}$ are positive,
so~that the constants $C_{\rm ev}^{(n)}[w^{(p,q)}]$ and $C_{\rm sv}^{(n)}[w^{(p,q)}]$
may always be defined by Eqs. \eqref{eq:CEV} and \eqref{eq:PED-normalization}.
Even more, Eq.~\eqref{eq:weight-6} implies that the interpolating eigenvalue densities 
are always non-negative.
 
\begin{proposition}[Interpolating Eigenvalue Density]
\label{prop:interpolating-EVD}
For any $p \ge 0$ and $q \ge 0$ with~$p+q > 0$, let $\FEV^{(p,q)}(z)$ be defined
by the right-hand side in Eq.~\eqref{eq:Interpolating-EVD}. Then $\FEV^{(p,q)}$ is non-negative, 
and hence a probability density on~$\mycmplx^n$.
\end{proposition}

\begin{proof}
By the preceding comments, it remains to show non-negativity. But this is immediate
from the representation \eqref{eq:weight-6} and the fact that the densities
$h_r(x)$ ($r > 0$) are non-negative.
\end{proof}
 
In contrast to this result, the interpolating singular value densities are non-negative
only within certain regions for the parameters.
  
\begin{proposition}[Interpolating Squared Singular Value Density]
\label{prop:interpolating-SVD}
For any $p \ge 0$ and $q \ge 0$ with $p+q > 0$, let $\FSV^{(p,q)}(\lambda)$ be defined 
by the right-hand side in Eq.~\eqref{eq:Interpolating-SVD}.

(i) If
$
(p \in \mynat_0 \text{ or } p > n-1) \text{ and } (q \in \mynat_0 \text{ or } q > n-1),
$
$\FSV^{(p,q)}$ is non-negative, and hence a probability density.

(ii) Otherwise, $\FSV^{(p,q)}$ is not a non-negative function.
\end{proposition}

Proposition \ref{prop:interpolating-SVD} should be compared to Gindikin's theorem
(see \eg \cite[Theorem VII.3.1]{FK}), which states that 
the $p$th \emph{additive} convolution power of the \emph{Wishart} dis\-tribution
is a probability distribution if and only if $p \in \mynat_0$ or $p > n-1$.

Since the proof of Proposition \ref{prop:interpolating-SVD} is a bit longer,
it is postponed to the end of this~section.

\medskip

\myparagraph{Interpolating Matrix Densities} 
In any case, regardless of whether the density $\FSV^{(p,q)}(\lambda)$ 
in Proposition~\ref{prop:interpolating-SVD} is non-negative or not,
we may consider the (possibly signed) bi-unitarily invariant matrix density 
\begin{align}
\label{eq:Interpolating-MD}
f_G^{(p,q)} := \mathcal{I}^{-1} \FSV^{(p,q)}
\end{align}
on $\GL(n,\mycmplx)$ with $\mathcal{I}$ as in Eq.~\eqref{I-def}.
Then, by construction, the induced squared singular value density is given by Eq.~\eqref{eq:Interpolating-SVD}. 
This density is non-negative only under~the~conditions 
set out in Proposition~\ref{prop:interpolating-SVD}\,(i).
Furthermore, by Theorem~\ref{thm:EigenvalueDistribution} extended to signed densities,
the induced eigenvalue density is given by Eq.~\eqref{eq:Interpolating-EVD},
which is always a non-negative density by Proposition~\ref{prop:interpolating-EVD}.
Thus, we obtain a large family of examples of probability densities on eigenvalues
for which the corresponding densities on squared singular values
(given by the SEV-transform $\mathcal{R}$ described in Subsection~\ref{subsec:MatrixSpaces})
are not probability densities. In particular, this means that these probability densities on eigenvalues 
cannot result from random matrices with bi-unitary invariance.

When the density~\eqref{eq:Interpolating-MD} is non-negative, \ie when
\begin{multline}
(p,q) \in W := \big\{ (p,q) \in \myreal^2 : p \ge 0,\,q \ge 0,\,p + q > 0 \text{ and } \\ (p \in \mynat_0 \text{ or } p > n-1) \text{ and } (q \in \mynat_0 \text{ or } q > n-1) \big\} \,,
\label{eq:NiceParameters}
\end{multline}
we call the resulting probability density \emph{interpolating matrix ensemble}.
Then we can prove the following result.

\begin{proposition}[Transfer for Interpolating Matrix Ensembles]
\label{prop:interpolating-convolution}
Let $X_1$ and $X_2$ be independent random matrices from the polynomial matrix ensembles of derivative type associated with the weight functions $w^{(p_1,q_1)}$ and $w^{(p_2,q_2)}$, respectively,
where $(p_1,q_1),(p_2,q_2) \in W$. Then the product $X_1 X_2$ is a random matrix from the polynomial matrix ensemble of derivative type associated with the weight function $w^{(p_1+p_2,q_1+q_2)}$.
\end{proposition}
 
\begin{proof}
This is an immediate consequence of Corollary \ref{cor:Transfer},
since $w^{(p_1,q_1)} \circledast w^{(p_2,q_2)} \linebreak[2] = w^{(p_1+p_2,q_1+q_2)}$
by Eq.~\eqref{eq:weight-2} and the basic properties of the Mellin transform.
\end{proof}

Proposition \ref{prop:interpolating-convolution} extends to the situation 
where the densities are not necessarily non-negative:
If $f_G^{(p_1,q_1)}(x)$ and $f_G^{(p_2,q_2)}(x)$ are (possibly signed) bi-unitarily invariant matrix densities 
as in Eq.~\eqref{eq:Interpolating-MD}, the multiplicative convolution
\begin{align}
\label{eq:MatrixConvolution}
(f_G^{(p_1,q_1)} \circledast f_G^{(p_2,q_2)})(x) := \int_G f_G^{(p_1,q_1)}(xy^{-1}) f_G^{(p_2,q_2)}(y) \, \frac{dy}{|\det y|^{2n}}
\end{align}
is equal to the (possibly signed) bi-unitarily invariant matrix density $f_G^{(p_1+p_2,q_1+q_2)}(x)$.
This can be proved in the same way as Proposition~\ref{prop:interpolating-convolution},
using the extension of Corollary~\ref{cor:Transfer} to signed densities.

Also, by Proposition \ref{prop:interpolating-convolution}, we get
a two-parameter family $\big\{ f_G^{(p,q)} : p,q \in W \big\}$
of probability densities on matrix space
which is closed under multiplicative con\-volution.
Clearly, this family could be extended to a \emph{convolution semigroup}
by adjoining the Dirac measure $\delta_{\eins}$ at the $n \times n$ identity matrix $\eins$
and by convolving probability measures instead of probability densities.

We now turn to the proof of Proposition \ref{prop:interpolating-SVD}.
To this end, we need some pre\-parations. Let $m \in \mynat$.
A function $f : \myreal \to \myreal$ is called a~\emph{P\'olya frequency function} 
of order $m$ (or $\PF_m$ for short) if it satisfies~\cite{Karlin:1968}
\begin{align}
\det(f(x_j - y_k))_{j,k=1,\hdots,\nu} \ge 0
\end{align}
for any $\nu=1,\hdots,m$ and any $x_1 < \hdots < x_\nu$, $y_1 < \hdots < y_\nu$. 
For example, it is well-known that for any $p > 0$, $l > 0$, the Gamma density
\begin{align}
\label{eq:GammaDensity}
f_{p,l}(x) := \frac{l^{\,p}}{\Gamma(p)} x^{p-1} e^{-lx} \, \pmb{1}_{(0,\infty)}(x) \,,
\end{align}
is $\PF_{\lfloor p+1\rfloor}$, where $\lfloor.\rfloor$ is the floor function,
meaning that $\lfloor p+1\rfloor$ is the largest integer smaller than or equal to $p+1$.
A rather elementary proof for this can be found in \cite[Chapter~3.2]{Karlin:1968}.
Alternatively, this follows from the relation
\begin{multline}
\det[(x_a-y_b)^{p-1}\pmb{1}_{(0,\infty)}(x_a-y_b)]_{a,b=1,\ldots,\nu}=\prod_{j=0}^{\nu-1}\frac{\Gamma(p)}{j!\Gamma(p-j)} \\
 \times \Delta_\nu(x) \Delta_\nu(y)\int_K[\det(x-kyk^*)]^{p-\nu} \, \pmb{1}_{{\rm Herm}_+}(x-kyk^*) \, d^*k>0,
\end{multline}
where $\pmb{1}_{{\rm Herm}_+}$ is the indicator function on the set of positive definite Hermitian $\nu\times\nu$ matrices. This relation was proven in Ref.~\cite{KKS:2015} for any integer $p\ge\nu$. However, it~is easy to see using Carlson's theorem~\cite[page 226]{Mehta} that it can be continued analytically up to the first non-integrable singularity, which means to any real $p>\nu-1$. For this goal we underline that the determinant is enforced to be positive by the indicator function. 
Additionally, we point out that $f_{p,l}$ may fail to be a P\'olya frequency function 
of order higher than $\lfloor p+1 \rfloor$. To see this, one can consider the limit
\begin{equation}
\lim_{y\to0}\frac{\det[(x_a-y_b)^{p-1}\pmb{1}_{(0,\infty)}(x_a-y_b)]_{a,b=1,\ldots,\nu}}{\Delta_\nu(y)}=\Delta_\nu(x)\prod_{j=0}^{\nu-1}\frac{\Gamma(p)x_{j+1}^{p-\nu}\pmb{1}_{(0,\infty)}(x_{j+1})}{j!\Gamma(p-j)}
\end{equation}
which is not always positive for $\nu > p+1$ due to the fact that the Gamma function can take also negative values on the negative half-line.
Indeed, this idea is very close to the~proof of Proposition~\ref{prop:interpolating-SVD}(ii) below.

Finally, we will need the facts that, for any $m \in \mynat$, 
the class $\PF_m$ is obviously closed under translations of arguments 
of the functions and even under additive convolutions, 
see also \cite[Prop.~7.1.5]{Karlin:1968}.
The latter property follows at once from the identity
\begin{multline}
\label{eq:PF-closure}
\int_{a_1\leq\ldots\leq a_\nu}\det(f_1(a_j - y_k))_{j,k=1,\hdots,\nu}\det(f_2(a_j - x_k))_{j,k=1,\hdots,\nu} \, da_1\cdots da_\nu\\
=\det\left(\int f_1(a - y_j)f_2(a - x_k) \, da\right)_{j,k=1,\hdots,\nu}
\end{multline}
for two functions $f_1,f_2\in \PF_{m}$.

\begin{proof}[Proof of Proposition \ref{prop:interpolating-SVD}]
By the comments above Propositions \ref{prop:interpolating-EVD} and \ref{prop:interpolating-SVD}, 
it re\-mains to address the issue of positivity.

To begin with, similarly as in Remark \ref{rem:InverseMatrix},
if $\FSV(\lambda)$ is the signed density associated with a function $w_0(x)$,
then the induced density under the inversion mapping 
$(\lambda_1,\hdots,\lambda_n) \mapsto (\lambda_1^{-1},\hdots,\lambda_n^{-1})$
is the signed density associated with the function 
$\widetilde{w}_0(x) := x^{-(n+1)} \, w_0(x^{-1})$.
Moreover, for the weight function in Eq.~\eqref{eq:weight-2}
we have $\widetilde{w}_0^{(p,q)}(x) = w_0^{(q,p)}(x)$.
Thus, it is possible to~prove the assertions
for certain restricted combinations of the parameters
and to extend the result ``by~inversion''.

{\it Part (i).} We prove the first statement that for $p,q\in\mathbb{N}_0\cup[n-1,\infty)$ the function $f_{\rm SV}^{(p,q)}$ is non-negative. It is sufficient to prove this claim for $p > 0$, $q = 0$, since the claim for $p > 0$, $q > 0$ 
then follows via the preceding comment and Proposition~\ref{prop:interpolating-convolution}.
Furthermore, in the case $p > 0$, $q = 0$ the claim is clear when $p \in \mynat$, $q = 0$,
since the function $\FSV(\lambda)$ then becomes the squared singular value density 
of the~product of $p$ Ginibre matrices, which is necessarily non-negative. 
Thus, it~remains to consider the case $p \not\in \mynat$, $p > n-1$, $q = 0$.

For $n = 1$, we have $w^{(p,0)}(x) = h_p(x)$, which is positive.
Hence, it~remains to consider the case where $n \ge 2$ and $p > n-1 \ge 1$.
Moreover, for~symmetry~reasons, it suffices to show that
\begin{align}
\label{eq:hq-determinant}
\det \left[ \left({-} x_k \frac{d}{d x_k}\right)^{j-1}  h_p(x_k) \right]_{j,k=1,\hdots,n} \ge 0
\quad\text{for $x_1 < \cdots < x_n$ in $(0,\infty)$,} 
\end{align}
or, equivalently,
\begin{align}
\label{eq:gq-determinant}
(-1)^{n(n-1)/2} \det\left[ \left(\frac{d}{d y_k}\right)^{j-1}  g_p(y_k) \right]_{j,k=1,\hdots,n} \ge 0 
\quad\text{for $y_1 < \cdots < y_n$ in $\myreal$.}
\end{align}
The equivalence of Eqs.~\eqref{eq:hq-determinant} and \eqref{eq:gq-determinant} 
follows from Eq.~\eqref{eq:gh-derivatives} as well as some straightforward row and column transformations.
The prefactor $(-1)^{n(n-1)/2}$ arises because the change of variables from $x$ to $y = -{\rm ln}\,x$ 
is decreasing, and hence must be compensated by inverting the order of the columns.

To~prove Eq.~\eqref{eq:gq-determinant}, we use a variation of an argument from \cite[page 388]{Karlin:1968}.
By~Eq.~\eqref{eq:gq-representation} as well as the functional equation and the product representation 
for the Gamma function, the Laplace transform of the density $g_p(x)$ is given by
\begin{align}\label{p.1}
\laplace g_p(s) = \Gamma^p(1+s) = e^{-\gamma ps} \prod_{l=1}^{\infty} \frac{e^{ps/l}}{(1+s/l)^p} \,,
\end{align}
where $\gamma = 0.577\ldots$ denotes the Euler constant.

We now start from the fact that, for any $l \in \mynat$, the~function $(1+s/l)^{-p}$ 
is the Laplace transform of the Gamma density $f_{p,l}(x)$, see Eq.~\eqref{eq:GammaDensity},
which is $\PF_{\lfloor p+1 \rfloor}$.
Thus, it follows that for any $m \in \mynat$,
the~density $g_{p,m}(x)$ corresponding to the Laplace transform
\begin{align}
\laplace g_{p,m}(s) = e^{-\gamma ps} \prod_{l=1}^{m} \frac{e^{ps/l}}{(1+s/l)^p}
\end{align}
is also $\PF_{\lfloor p+1\rfloor}$, being a finite additive convolution
of $\PF_{\lfloor p+1 \rfloor}$ functions.

Next, as $m \to \infty$, 
we clearly have $\laplace g_{p,m}(\imath t) \to \laplace g_p(\imath t)$ 
pointwise in $t \in \myreal$ as well as 
\begin{align}
  \big| \laplace g_{p,m}(\imath t) \big| 
= \prod_{l=1}^{m} \frac{1}{|1+\imath t/l|^p}
\leq \frac{1}{(1+t^2)^{p/2}}
\end{align}
for all $t \in \myreal$. 
Therefore, as $p > 1$, it follows from the Laplace inversion formula
and the dominated convergence theorem that
$g_{p,m}(x) \to g_p(x)$ pointwise in $x \in \myreal$.
Thus, the limit function $g_p(x)$ is also $\PF_{\lfloor p+1\rfloor}$, i.e.\@ 
for any $\nu=1,\hdots,{\lfloor p+1\rfloor}$ and any $x_1 < \hdots < x_\nu$, $y_1 < \hdots < y_\nu$,
we have
\begin{align}
\det(g_p(x_j - y_k))_{j,k=1,\hdots,\nu} \ge 0 \,.
\end{align}
Finally, letting $y_k \to 0$, see \eg \cite[page 7]{KS:1966} for the details of the argument, 
we~find that, for any $\nu=1,\hdots,{\lfloor p+1\rfloor}$ and any $x_1 < \hdots < x_\nu$,
\begin{align}
(-1)^{\nu(\nu-1)/2} \det\left(\frac{d^{k-1}}{dx_j^{k-1}} g_p(x_j)\right)_{j,k=1,\hdots,\nu} \ge 0 \,.
\end{align}
Since ${\lfloor p+1\rfloor} \geq n$, this proves Eq.~\eqref{eq:gq-determinant}.

{\it Part (ii).} Now we turn to the proof of the second claim of Prop.~\ref{prop:interpolating-SVD} that whenever $p$ or $q$ is not in $\mathbb{N}_0\cup[n-1,\infty)$ the function $f_{\rm SV}^{(p,q)}$ is not non-negative.
As explained at the beginning of the proof,
it~suffices to consider the case where $p \in [0,n-1]$ is not an~integer and $q \ge 0$.
We will show that there exists a point $x = (x_1,\hdots,x_n)$ 
with $0 < x_1 < \hdots < x_n < \infty$ 
and most coordinates $x_k$ very close to~zero such that $\FSV^{(p,q)}(x) < 0$.
Observe that for a point $x$ of this form, 
the sign of~$\FSV^{(p,q)}(x)$ is equal to~that of
$\det ( w^{(p,q)}_{j-1}(x_k))_{j,k=1,\hdots,n}$.
To obtain the asymptotic behavior of $w_{j-1}^{(p,q)}(x)$ as $x \to 0$ or $x \to \infty$, 
we will use the following fact.

\begin{theorem*}[Special Case of {\cite[Theorem 7.3.1]{Doetsch2}}] 
For $a,b \in \myreal$ with $a < b$, let $F(s)$ be an analytic function on the set $([a,b] + \imath \myreal) \setminus \{ a,b \} \subset \mycmplx$ 
which satisfies the following assumptions:
\begin{enumerate}[(a)]
\item 
there exists a representation $F(s) = \sum_{k=0}^{K} c_k (s-a)^{\alpha+k} + G(s)$
with $\alpha \in \myreal$, $G(s) \in \mathscr{C}^L([a,b[ + \imath \mathbb{R})$
$($the space of functions on $[a,b[ + \imath \mathbb{R}$ 
which are $L$-times continuously differentiable$)$,
and $K,L \in \mynat_0$,
\item
$F(x + \imath y) \to 0$ as $|y| \to \infty$, uniformly in $a \leq x \leq b$,
\item
for each $l=0,\hdots,L-1$, $F^{(l)}(a+\imath y) \to 0$ as $|y| \to \infty$,
\item
the integrals $\int_{+1}^{+\infty} e^{\imath ty} \, F^{(L)}(a+\imath y) \, dy$ 
and $\int_{-\infty}^{-1} e^{\imath ty} \, F^{(L)}(a+\imath y) \, dy$ 
(viewed as improper Riemann integrals)
are uniformly convergent for $t \geq 1$.
\end{enumerate}
Then, for any $c \in (a,b)$, we have 
\begin{align}
\label{eq:LA}
f(t) := \frac{1}{2\pi\imath} \int_{c-\imath\infty}^{c+\imath\infty} F(s) \, e^{ts} \, ds 
= e^{at} \sum_{k=0}^{K} \frac{c_k \, t^{-\alpha-k-1}}{\Gamma(-\alpha-k)} 
+ o(t^{-L} \, e^{at})
\end{align}
as $t \to \infty$, where $1/\Gamma(-\alpha-k) := 0$ when $\alpha+k \in \mynat_0$.
\end{theorem*}

Let us emphasize that due to the last convention, this result is true for all values $\alpha \in \myreal$.
Indeed, in the representation for $F(s)$ in part (a), the summands with $\alpha + k \in \mynat_0$
may be absorbed into the function $G(s)$.
In particular, when $\alpha \in \mynat_0$, the conclusion \eqref{eq:LA} 
reduces to the statement that $f(t) = o(t^{-L} \, e^{at})$ as $t \to \infty$.

Moreover, let us mention that by applying the preceding theorem to $F(-s)$,
we may also obtain, under the appropriate assumptions, the asymptotic behavior 
of $f(t)$ as $t \to -\infty$. Finally, let us recall that $\Gamma(z)$ has a simple pole with residue $1$ at the origin.

It follows from our definitions that, for each $j=1,\hdots,n$,
\begin{align}
   w_{j-1}^{(p,q)}(x) 
&= \biggl({-}x \frac{d}{dx}\biggl)^{j-1} w^{(p,q)}(x) \nonumber \\
&= \frac{1}{2\pi\imath} \int_{c-\imath\infty}^{c+\imath\infty} s^{j-1} \, \Gamma^{p}(s) \Gamma^q(1+n-s) \, x^{-s} \, ds \,, \quad x > 0 \,,
\end{align}
where $0 < c < n+1$.

To obtain the asymptotics of $w_{j-1}^{(p,q)}(x)$ as $x \to 0$, 
we set $x = e^{-t}$ and consider $w_{j-1}^{(p,q)}(e^{-t})$ as $t \to \infty$. 
Using the above result with $a := 0$, $b := 1+n$,
$\alpha := j-1-p$, $F(s) := s^{\alpha} (s^p \, \Gamma^p(s) \, \Gamma^q(1+n-s))$
and $K,L \in \mynat_0$ such that $\max \{ \alpha+1,0 \} \leq \alpha+K \le L < \alpha+K+1$, we~obtain
\begin{align}
\label{eq:asymptotics-2}
w_{j-1}^{(p,q)}(e^{-t}) = \sum_{k=0}^{K} \frac{c_k \, t^{p-j-k}}{\Gamma(p-j-k+1)} + o(t^{-L})
\end{align}
as $t \to \infty$, where
\begin{equation}
 c_k=\frac{1}{k!}\left.\frac{d^k}{ds^k}\left(s^p\Gamma^p(s)\Gamma^q(1+n-s)\right)\right|_{s=0},\ k=0,\ldots,K.
\end{equation}
Therefore, we have
\begin{align}
\label{eq:asymptotics-3}
w_{j-1}^{(p,q)}(x) = \frac{\Gamma^q(1+n) \, (- {\rm ln}\, x)^{p-j}}{\Gamma(p-j+1)} + o((- {\rm ln}\, x)^{p-j})
\end{align}
as $x \to 0$. \pagebreak[1]

For the asymptotics of $w_{j-1}^{(p,q)}(x)$ as $x \to \infty$, 
we set $x = e^{t}$ and change the integration variable $s \to 1+n-s$
such that the function reads
\begin{align}
\label{eq:asymptotics-4}
w_{j-1}^{(p,q)}(e^{t}) = e^{-(n+1)t} \, \frac{1}{2\pi\imath} \int_{c-\imath \infty}^{c+\imath \infty} (1\!+\!n\!-\!s)^{j-1} \, \Gamma^{p}(1\!+\!n\!-\!s) \, \Gamma^q(s) \, e^{ts} \, ds
\end{align}
as $t \to \infty$. Again we can apply the above theorem, 
but now with the identification $a := 0$, $b := 1+n$,
$\alpha := -q$, $F(s) := s^\alpha ( s^q \, \Gamma^q(s) \, (1+n-s)^{j-1} \, \Gamma^p(1+n-s) )$
and $K,L \in \mynat_0$ such that $\max \{ \alpha+1,0 \} \leq \alpha + K \leq L < \alpha + K + 1$.
Thus, we obtain the expansion
\begin{align}
\label{eq:asymptotics-5}
w_{j-1}^{(p,q)}(e^t) = e^{-(n+1)t} \, \sum_{k=0}^{K} \frac{c'_k \, t^{q-k-1}}{\Gamma(q-k)} + o(t^{-L} e^{-(n+1)t})
\end{align}
as $t \to \infty$, with the coefficients
\begin{equation}
 c'_k=\frac{1}{k!}\left.\frac{d^k}{ds^k}\left(s^q \, \Gamma^q(s) \, (1+n-s)^{j-1} \, \Gamma^p(1+n-s)\right)\right|_{s=0},\ k=0,\ldots,K.
\end{equation}
This yields
\begin{align}
\label{eq:asymptotics-6}
w_{j-1}^{(p,q)}(x) = \frac{(1+n)^{j-1} \, \Gamma^p(1+n) \, ({\rm ln}\, x)^{q-1}}{\Gamma(q) \, x^{n+1}} + o\left(\frac{({\rm ln}\, x)^{q-1}}{x^{n+1}}\right)
\end{align}
as $x \to \infty$. \pagebreak[1]

In the next step we consider the vector $x = (x_1,\hdots,x_n)$ with a fixed $x_n > 0$ 
such that $w_{n-1}^{(p,q)}(x_n) \ne 0$
and $x_k := e^{-\alpha/k}$, $k=1,\hdots,n-1$, for $\alpha > 0$ sufficiently large. The precise choice of $x_n$ will be specified later.
Then, we make use of the asymptotics~\eqref{eq:asymptotics-3} and plug it into the determinant
\begin{align}
\label{eq:asymptotics-8}
\det(w_{j-1}^{(p,q)}(x_k))_{j,k=1,\hdots,n} = w_{n-1}^{(p,q)}(x_n) \frac{C\alpha^\kappa}{\prod_{j=1}^{n-1} \Gamma(p-j+1)} + o(\alpha^{\kappa})
\end{align}
as $\alpha \to \infty$.
The exponent is $\kappa := (n-1)p-(n-1)(n-2)/2$, 
$C$ is a positive constant not depending on $\alpha$ or $x_n$,
and the implicit constant in the $o$-bound may depend on $x_n$.

In Eq.~\eqref{eq:asymptotics-8}, the numerator of the fraction is positive,
whereas the denominator of the fraction may be positive or negative,
depending on the number of negative Gamma factors.
Moreover, since $p < n-1$, the asymptotics~\eqref{eq:asymptotics-3} and \eqref{eq:asymptotics-6} for $j = n-1$ 
show that
\begin{align}
w_{n-2}^{(p,q)}(x) = C_1 (- {\rm ln}\, x)^{p-n+1} + o((- {\rm ln}\, x)^{p-n+1}) \overset{x\to0}{-\!\!\!\longrightarrow} 0
\end{align}
and 
\begin{align}
w_{n-2}^{(p,q)}(x) = C_2 ({\rm ln}\, x)^{q-1} / x^{n+1} + o(({\rm ln}\, x)^{q-1} / x^{n+1}) \overset{x\to\infty}{-\!\!\!-\!\!\!\longrightarrow} 0.
\end{align}
Here, $C_1$ and $C_2$ are real constants with $C_1 \ne 0$. These asymptotics in combination with the continuity of $w_{n-2}^{(p,q)}(x)$ for all $x>0$ imply that $w_{n-1}^{(p,q)}(x) = (- x \frac{d}{dx}) w_{n-2}^{(p,q)}(x)$ must change its sign 
on the interval $(0,\infty)$, since $w_{n-2}^{(p,q)}$ is not equal to zero.
Thus, we~may choose $x_n > 0$ in such a way that $w_{n-1}^{(p,q)}(x_n)$ is of sign opposite 
to that of the fraction in~Eq.~\eqref{eq:asymptotics-8}. Therefore Eq.~\eqref{eq:asymptotics-8} 
gets negative as $\alpha \to \infty$, which is what we needed to show.
\end{proof}

\section{Central Limit Theorem for Lyapunov Exponents}
\label{sec:Large-M}

Finally we want to make contact to the limiting distribution of the singular~values of the product $X^{(M)}=X_1\cdots X_M$ of independently and identically distributed (i.i.d.\@), bi-unitarily invariant $n\times n$ random matrices $X_j$ when $M\to\infty$ and the dimension $n$ is fixed. Let $\sigma_{M,1} \ge \cdots \ge \sigma_{M,n}$ and $|\lambda_{M,1}| \ge \cdots \ge |\lambda_{M,n}|$ denote the ordered singular values and the radii of the eigenvalues of $X^{(M)}$. Then the quantities
$\tfrac1M \ln \sigma_{M,k}$ and $\tfrac1M \ln |\lambda_{M,k}|$ are also called the (\emph{finite}) \emph{Lyapunov exponents} and the (\emph{finite}) \emph{stability exponents} of the matrix $X^{(M)}$. Strictly speaking, the~names \emph{Lyapunov exponents} and \emph{stability exponents} refer to the limits of the quantities above as $M \to \infty$, but we will use these names exclusively for the non-asymptotic quantities here.

The investigation of limit theorems for Lyapunov exponents has a long history,
see \eg \cite{Sazonov-Tutubalin:1966}, \cite{Bougerol-Lacroix:1985}
and the references therein. In particular, for the Gaussian matrix ensemble, 
Newman~\cite{Newman:1986} provided a simple argument that the Lyapunov exponents 
converge almost surely to deterministic values which are equal~to the expected values 
of the logarithms of $\sum_{j=1}^{n-k+1} |X_{j1}^{(1)}|^2$, $k=1,\hdots,n$.
In the past few years, new interest has arisen in this limit due to explicit results 
for the joint densities of the singular value and the eigenvalues at finite $n$ and $M$, 
see \cite{ABK:2014,Forrester:2013,Forrester:2015,Ipsen:2015,Kargin:2014,Reddy:2016}. 
In particular, the very recent work~\cite{Reddy:2016} contains a general result 
on the Lyapunov and stability exponents which we cite here.

\begin{proposition}[Central Limit Theorem for Lyapunov and Stability Exponents {\cite[Theorem 11]{Reddy:2016}}] \label{prop:CLT}

Let $(X_j)_{j \in \mynat}$ be a sequence of independently and identically distributed, bi-unitarily invariant $n\times n$ random matrices
such that, with positive probability, $X_1$ is not a scalar multiple of the identity matrix. 
Moreover, suppose that the first and second moments of the logarithms 
of the singular values of $X_1$ are finite, and let the vector $m = (m_k)_{k=1,\hdots,n}$ and the matrix $\Sigma = (\Sigma_{jk})_{j,k=1,\hdots,n}$ be the mean and the covariance matrix of the random vector $(\ln R_{11},\hdots,\ln R_{nn})$, where $(R_{11},\hdots,R_{nn})$ is the diagonal of the matrix $R$ in the $QR$-decomposition $X_1 = QR$ of the random matrix $X_1$.
Then, as $M \to \infty$ with $n$ fixed, the random vectors
\begin{equation}
\left\{ \sqrt{M} \left( \frac1M \ln \sigma_{M,k} - m_k \right) \right\}_{k=1,\hdots,n}
\ \text{and}\ 
\left\{ \sqrt{M} \left( \frac1M \ln |\lambda_{M,k}| - m_k \right) \right\}_{k=1,\hdots,n}
\end{equation}
converge in distribution to a Gaussian random vector with mean $0$ and covariance matrix $\Sigma$.
\end{proposition}

For completeness, we emphasize that Proposition~\ref{prop:CLT} is stated and proved  not only for products of bi-unitarily invariant (complex) matrices, but also for products of bi-orthogonally invariant (real) matrices in \cite{Reddy:2016}. Moreover, one can also random\-ize the order of the singular values and of the eigenvalues of $X^{(M)}$, which naturally arises in the alternative derivation of Proposition~\ref{prop:CLT} we present below. This~sym\-metrized version was also derived in Refs.~\cite{ABK:2014,Ipsen:2015} for the induced Laguerre ensemble. Since the components of the vector $m$ are pairwise different \cite{Reddy:2016}, both representations, ordered or not, are equivalent and only a matter of taste.

As pointed out in \cite{Reddy:2016}, in the complex case, it turns out that for products of Ginibre or truncated unitary matrices, the (ordered) Lyapunov exponents and the (ordered) stability exponents are asymptotically independent. Our aim is to show that this result remains true for any polynomial ensemble of derivative type.

\begin{proposition} \label{prop:CLT-DPE}
Suppose that, in the situation of Proposition \ref{prop:CLT}, the random matrices $X_j$ are taken from a polynomial ensemble of derivative type. 
Then the covariance matrix $\Sigma$ is diagonal, i.e.\@ the Lyapunov and stability exponents are asymptotically independent.
\end{proposition}

To obtain this result, we will use another way of deriving Proposition~\ref{prop:CLT} in the complex case, viz.\@ by means of the spherical transform. The interest in this approach is due to the fact that, at least in principle, it should be able to yield much more precise information about the convergence in distribution, similarly to the Fourier transform in classical probability theory. 
In fact, the spherical transform has been used to derive a variety of central limit theorems on matrix spaces, with different assumptions and different scalings, see \eg \cite{Sazonov-Tutubalin:1966,Bougerol:1981,Terras:1984,Terras,Richards:1989,GL:1995,Voit} and the references therein.
Although this approach seems to be well known, we include a~sketch of a proof for didactic purposes
and for the sake of completeness. In addition to that, our discussion will also lead to some new observations which shed some light into the recent developments in the field of random matrix theory, see Proposition \ref{prop:CLT-DPE} above and Corollary \ref{cor:Characterization} below.

To keep the presentation short, we will restrict ourselves to a particularly simple situation.
However, let us mention that it is possible to adapt the proof to the more general situation 
of Proposition \ref{prop:CLT} by using a few additional arguments. We assume that the random matrix 
$X_1$ has a bi-unitarily invariant density $f_G$ with respect to Lebesgue measure. 
Also, we assume that the spherical transform $S_{X_1}(z)$ decays sufficiently fast 
as $z \to \infty$ in $\operatorname{conv}(\mathbb{S}_n(\varrho')) + \imath \mathbb{R}^n$,
where $\operatorname{conv}(\mathbb{S}_n(\varrho'))$ denotes the convex hull of the orbit 
of the vector $\varrho'$ from Eq.~\eqref{rho-def} under the natural action of the symmetric group 
$\mathbb{S}_n$. In particular, this~implies that the function $S_{X_1}^{M}(\varrho' + \imath s) \Delta_n(\varrho' + \imath s)$ is integrable in $s \in \mathbb{R}^n$ for all $M \in \mynat$ 
large enough.

We start with the multiplication theorem~\eqref{eq:SphericalMultiplicationX} 
for the spherical transform, which implies that $S_{X^{(M)}}=S_{X_1}^M$. 
Let $\FSV^{(M)} := \mathcal{I}(f_G^{\circledast M})$ be the joint density of the squared singular values of $X^{(M)}$, where $f^{\circledast m}$ denotes the $m$-fold convolution power of $f$ and $\mathcal{I}$ denotes the operator defined in Eq.~\eqref{I-def}. Then, by spherical inversion, see~\eg~\cite[Lemma 2.10]{KK:2016a} and note that the regularization there may be omitted by our previous assumptions, we~have
\begin{multline}
 \FSV^{(M)}(\lambda) = \frac{1}{(n!)^2\prod_{j=0}^{n-1}j!}\Delta_n(\lambda)\int_{\mathbb{R}^n} (S_{X_1}(\varrho'+\imath s))^M \\
 \times\Delta_n(\varrho'+\imath s)\det[\lambda_b^{-c-\imath s_c}]_{b,c=1,\ldots,n}\prod_{j=1}^n\frac{ds_j}{2\pi} \,.
\end{multline}
To prepare for the limit $M\to\infty$, we substitute $\lambda=\tilde\lambda^M$,
which yields the density
\begin{multline}
 \FSV^{(M)}(\tilde\lambda^M)M^n\det \tilde\lambda^{M-1} =\frac{M^n\det \tilde\lambda^{M-1}}{ (n!)^2\prod_{j=0}^{n-1}j!}\Delta_n(\tilde\lambda^M) \int_{\mathbb{R}^n}
 (S_{X_1}(\varrho'+\imath s))^M \\\times
 \Delta_n(\varrho'+\imath s)\det[\tilde\lambda_b^{-M(c+\imath s_c)}]_{b,c=1,\ldots,n}\prod_{j=1}^n\frac{ds_j}{2\pi}.
\end{multline}
The prefactor $M^n\det \tilde\lambda^{M-1}$ is the Jacobian of the substitution. Expanding the~two determinants $\Delta_n(\tilde\lambda^M) = \det[\tilde\lambda_b^{M(c-1)}]$ and $\det[\tilde\lambda_b^{-M(c+\imath s_c)}]$, we find that
\begin{multline}
 \FSV^{(M)}(\tilde\lambda^M) M^n\det \tilde\lambda^{M-1} = \frac{M^n\det \tilde\lambda^{-1}}{(n!)^2\prod_{j=0}^{n-1}j!}\sum_{\sigma_1,\sigma_2\in\mathbb{S}_n}{\rm sign}\,\sigma_1\sigma_2 \\\times 
\int_{\mathbb{R}^n} (S_{X_1}(\varrho'+\imath s))^M 
\Delta_n(\varrho'+\imath s)\prod_{j=1}^n\tilde\lambda_{\sigma_2^{-1}(j)}^{-M(j+\imath s_j-\sigma_1\sigma_2^{-1}(j))}\frac{ds_j}{2\pi},
\label{eq:result-4}
\end{multline}
where $\mathbb{S}_n$ is the symmetric group of $n$ elements and ``${\rm sign}$" is $+1$ for even permutations and $-1$ for odd ones. 
Now, by our simplifying assumptions, we can make the~substitution $\imath s\to \imath s/M-\varrho'+\sigma_1\sigma_2^{-1}(\varrho')$, where $(\sigma_1\sigma_2^{-1}(\varrho'))_j=\varrho'_{\sigma_1\sigma_2^{-1}(j)}$, and pull the domain of integration back to $\myreal^n$. This is possible because, for $f_G \in L^{1,K}(G)$, the spherical transform $\mathcal{S} f_G(s)$ is analytic in the interior of the~set $\operatorname{conv}(\mathbb{S}_n(\varrho')) + \imath \myreal^n$ when regarded as a function of $n-1$ of the variables $s_j$, with the sum $\sum_{j=1}^{n} s_j$ held fixed;
compare \cite[Theorem IV.8.1 and Corollary IV.8.2]{Helgason3}.
Thus, we obtain
\begin{multline}
 \FSV^{(M)}(\tilde\lambda^M)M^n\det \tilde\lambda^{M-1}
 =\frac{\det \tilde\lambda^{-1}}{(n!)^2\prod_{j=0}^{n-1}j!}\sum_{\sigma_1,\sigma_2\in\mathbb{S}_n}{\rm sign}\,\sigma_1\sigma_2 \\\times \int_{\mathbb{R}^n} \left[S_{X_1}\left(\sigma_1\sigma_2^{-1}(\varrho')+\imath \frac{s}{M}\right)\right]^M
\Delta_n\left(\sigma_1\sigma_2^{-1}(\varrho')+\imath \frac{s}{M}\right) \prod_{j=1}^n\tilde\lambda_{\sigma_2^{-1}(j)}^{-\imath s_j}\frac{ds_j}{2\pi}.
\label{eq:result-5}
\end{multline}
In the next step, we expand the logarithm of the spherical transform in a Taylor series up to the second order in $1/M$, which requires certain assumptions. Since the spherical transform is symmetric in its arguments, it is sufficient to specify the expansion around the point $\sigma'(\varrho')=\{\varrho'_{\sigma'(j)}\}_{j=1,\ldots,n}$, where the permutation $\sigma'\in\mathbb{S}_n$ is given by $\sigma'(j)=n-j+1$, $j=1,\hdots,n$. Then the first derivatives may be expressed in terms of the vector
\begin{equation}
\label{eq:M}
m= - \frac{\imath}{2} \left\{\left.\frac{\partial}{\partial s_a}{\rm ln}\,S_{X_1}\left(\imath s +\sigma'(\varrho')\right)\right|_{s=0}\right\}_{a=1,\hdots,n},
\end{equation}
and the second derivatives in terms of the matrix
\begin{equation}
\label{eq:S}
\Sigma= - \frac{1}{4} \left\{\left.\frac{\partial^2}{\partial s_a\partial s_b}{\rm ln}\,S_{X_1}\left(\imath s+\sigma'(\varrho')\right)\right|_{s=0}\right\}_{a,b=1,\ldots,n},
\end{equation}
where all derivatives are real derivatives. It will eventually turn out that only this particular definition with $\sigma'(\varrho')$ instead with $\varrho'$ is consistent with that in Proposition~\ref{prop:CLT}. Moreover, it will follow from this connection that $m$ and $\Sigma$ have real components, the components of the vector $m$ are in strictly decreasing order, and the matrix $\Sigma$ is positive-definite. Thus the Taylor approximation reads
\begin{align}
\left[S_{X_1}\left(\sigma'(\varrho')+\imath \frac{s}{M}\right)\right]^M
&=
\exp \Big( 2\imath {s}^T m - \frac{2}{M} {s}^T \Sigma {s} + |s|^2 \, o(1/M) \Big) \nonumber\\
&=\exp \Big( 2\imath {s}^T m - \frac{2}{M} {s}^T \Sigma {s} \Big)\Big(1 + |s|^2 \, o(1/M) \Big) 
\label{eq:taylor}
\end{align}
as $M \to \infty$, where the symbol $(\,\cdot\,)^T$ denotes transposition and the $o$-bound holds uniformly in $s$ in any ball of radius $\mathcal{O}(\sqrt{M})$ around the origin.
Also the Vandermonde determinant in Eq.~\eqref{eq:result-5} can be expanded in $1/M$. 

Then, as $M \to \infty$, by our simplifying assumptions, we are left with a Gaussian integral in~$s$ which can be performed explicitly,
\begin{multline}
 \FSV^{(M)}(\tilde\lambda^M)M^n\det \tilde\lambda^{M-1} = \frac{M^{n/2}\det \tilde\lambda^{-1}}{(n!)^2\prod_{j=0}^{n-1}j!} \sum_{\sigma_1,\sigma_2\in\mathbb{S}_n}{\rm sign}\,\sigma_1\sigma_2  \frac{\Delta_n\left(\sigma_1\sigma_2^{-1}(\varrho')\right)}{(2\pi)^{n/2}\sqrt{\det(4\Sigma)}}\\
 \times\exp\left(-\frac{M}{8}(2m-\ln \sigma_1^{-1} \sigma'(\tilde\lambda))^T \Sigma^{-1}(2m-\ln\sigma_1^{-1} \sigma'(\tilde\lambda)) \right) + o(M^{n/2}) \,,
 \label{eq:result-6}
\end{multline}
where ${\rm ln}\,\sigma(\tilde\lambda)$ is a short-hand notation for $\{{\rm ln}\,\tilde\lambda_{\sigma(a)}\}_{a=1,\hdots,n}$ and the $o$-bound holds locally uniformly in $\widetilde\lambda$.

Since $\Delta_n\left(\sigma_1\sigma_2^{-1}(\varrho')\right) = {\rm sign}\,\sigma_1\sigma_2^{-1}\ \prod_{j=0}^{n-1} j!$ and the permutation $\sigma'$ may be absorbed into $\sigma := \sigma_1^{-1} \sigma'$ due to the summation over $\sigma_1$, the result simplifies to
\begin{multline}
 \FSV^{(M)}(\tilde\lambda^M)M^n\det \tilde\lambda^{M-1}= \frac{M^{n/2}}{(2\pi)^{n/2}n!}\frac{1}{\sqrt{\det(4\Sigma)}\det \tilde\lambda} \\
 \times \sum_{\sigma\in\mathbb{S}_n}
\exp\left(-\frac{M}{8}(2m-{\rm ln}\,\sigma(\tilde\lambda))^T\Sigma^{-1}(2m-{\rm ln}\,\sigma(\tilde\lambda))\right) + o(M^{n/2}) \,.
\label{eq:result-7}
\end{multline}
When substituting $y :=\ln (\tilde\lambda)/2$, we can identify $m$ and $ \Sigma/M$ as the mean and the covariance matrix of the asymptotic Gaussian distribution of the Lyapunov exponents of $X^{(M)}$. Thus, we arrive at the unordered version of the first part of Prop.~\ref{prop:CLT}.

For polynomial ensembles of derivative type, the spherical transform factorizes by Eq.~\eqref{eq:SphericalTransform-12}, from which it follows that the matrix $\Sigma$ is diagonal, and the previous result simplifies further to a permanent
\begin{multline}
 \FSV^{(M)}(\tilde\lambda^M)M^n\det \tilde\lambda^{M-1} \\ = \frac{1}{n!} \, {\rm perm}\left[\sqrt{\frac{M}{8\pi\Sigma_{aa}}}\frac{1}{\tilde\lambda_b} \exp\left(-\frac{M}{8\Sigma_{aa}}(2m_a-{\rm ln}\,\tilde\lambda_b)^2\right) \right]_{a,b=1,\ldots,n} + o(M^{n/2}) \,,
\label{eq:result-8}
\end{multline}
which is the first part of Prop.~\ref{prop:CLT-DPE}.
This result was already derived for particular Meijer G-ensembles, 
especially the induced Wishart-Laguerre ensemble, see \cite{ABK:2014,Ipsen:2015}.

It remains to justify the Taylor expansion \eqref{eq:taylor} and to clarify the relationship between Eqs.~\eqref{eq:M} and \eqref{eq:S} and the moments in Proposition~\ref{prop:CLT}. This will follow from the following lemma, which will also be useful for the discussion of the stability exponents 
as well as for Corollary \ref{cor:Characterization}.

\begin{lemma}
\label{lemma:R-MellinTransform}
Let $X$ be a bi-unitarily invariant random matrix with a density \linebreak $f_G \in L^{1,K}(G)$,
let $X = QR$ be its $QR$-decomposition, and let $(R_{11},\hdots,R_{nn})$ be the diagonal of the matrix $R$.
Then
\begin{equation}
\label{eq:SM}
\mys_{X}(s) = \mellin_{(R_{11}^2,\hdots,R_{nn}^2)}(s -\sigma'(\varrho')+ \mathbf{1}_n),
\end{equation}
where $\mathbf{1}_n := (1,\hdots,1)$, $\sigma'(\varrho') = (\varrho'_{n-j+1})_{j=1,\ldots,n}$
with $\varrho_j'$ as in Eq.~\eqref{rho-def},
and $\mellin_{(R_{11}^2,\hdots,R_{nn}^2)}(s_1,\hdots,s_n) := \ee \left( \prod_{j=1}^{n} R_{jj}^{2(s_j-1)} \right)$
is the multivariate Mellin transform of the random vector $(R_{11}^2,\hdots,R_{nn}^2)$.
\end{lemma}

\begin{proof}
We use another well-known representation of the spherical transform,
\begin{align}\label{Slimit-rel.c}
 Sf_G(s)= 2^n C_n^* \int_A \int_T f_G(at) \, dt \left( \prod_{j=1}^{n} a_j^{2s_j + 2\varrho_j' - 4n} \right)  \left( \prod_{j=1}^{n} a_j^{4(n-j) + 1} \right) \, da \,, 
\end{align}
for $f_G\in L^{1,K}(G)$; compare e.g.\@ \cite[Eq.~(IV.2.7)]{Helgason3}. This is based on the QR~de\-composition (or Iwasawa decomposition) of an element $g \in G$ in the form $g = k a t$ with $k \in K$, $a \in A$ and $t \in T$; see Subsection \ref{subsec:MatrixSpaces} for the definitions of these spaces and their reference measures. 

Note that the random vector $(R_{11},\hdots,R_{nn})$ corresponds to the main diagonal of the diagonal matrix $a$. Thus, its density can be obtained by setting $s_j := -\varrho_j' + 2n$ and integrating over $t\in T$, which yields
\begin{align}\label{Slimit-rel.d}
h(a) := 2^n C_n^*  \left( \prod_{j=1}^{n} a_j^{4(n-j) + 1} \right)\int_T f_G(at) \, dt \,, \qquad a \in A \,.
\end{align}
We therefore obtain
\begin{align}
Sf_G(s) &= \int_A h(a) \prod_{j=1}^{n} a_j^{2s_j + 2\varrho_j' - 4n} \, da \nonumber \\
        &= \ee \left( \prod_{j=1}^{n} R_{jj}^{2s_j + 2\varrho_j' - 4n} \right)
        = \mellin_{(R_{11}^2,\hdots,R_{nn}^2)}(s -\sigma'(\varrho')+ \mathbf{1}_n ) \,,
\end{align}
and the proof is complete.
\end{proof}

It follows from the preceding lemma that
\begin{multline}\label{rel-der-mom}
  \frac{d}{ds_j} \mys f_G(\imath s+\sigma'(\varrho') ) \bigg|_{s = 0}
= \frac{d}{ds_j} \mellin_{(R_{11}^2,\hdots,R_{nn}^2)}( \imath s+\mathbf{1}_n ) \bigg|_{s = 0} = 2\imath \, \ee(\ln R_{jj}) \,,
\end{multline}
which shows that our definition of the vector $m$ in Eq.~\eqref{eq:M} is consistent with that in Proposition~\ref{prop:CLT}, and a similar comment applies to the matrix $\Sigma$ in Eq.~\eqref{eq:S}. Furthermore, using the well-known properties of the Fourier transform, we now see that our differentiability assumptions on the spherical transform needed for Eq.~\eqref{eq:SM} are equivalent to the existence of the first and second moments of the random variables $\ln R_{ii}$, $i=1,\hdots,n$, and it is easy to  see that this is equivalent to the existence of the  first and second moments of the random variables $\ln \sigma_{1,i}$, $i=1,\hdots,n$, as~required in Theorem \ref{prop:CLT}. This concludes our discussion of the Lyapunov exponents. 

As regards the stability exponents, we start from the observation that the joint density $f_{\text{EVM}}(a)$ of the radii of the eigenvalues of a random matrix $X$ with a bi-unitarily invariant matrix density $f_G$ is given by 
\begin{align}
   f_{\text{EVM}}(a)
&= (2\pi)^n (\det a) \int_{({\rm U}(1))^n} \FEV(a \Phi) \, d^*\Phi \nonumber \\ 
&= \left( 2^n C_n^* \int_T f_G^{}(a t) \, dt \, \prod_{j=1}^{n} a_j^{n-2j+1} \right)
	\, (\det a)^n \, \frac{1}{n!} \text{perm}\Big(a_j^{2(k-1)}\Big)_{j,k=1,\hdots,n} \,,
\label{eq:EigenvalueModulusDensity}
\end{align}
where $a = \diag(a_1,\hdots,a_n) \in A$,
$({\rm U}(1))^n$ is the $n$-fold direct product of the unitary group ${\rm U}(1)$,
$\Phi = \text{diag}(e^{\imath \varphi_1},\hdots,e^{\imath \varphi_n})$,
$d^* \Phi$ denotes integration by the normalized Haar measure on $({\rm U}(1))^n$,
``$\text{perm}$" denotes the permanent, and the equality of the first and second line 
follows by combining Eq.~(2.19) and Lemma~2.7 in \cite{KK:2016a}.
The expression in the large round brackets is essentially the \emph{Harish transform}
\cite{JL:SL2R} of the function $f_G$, which is known to be symmetric in $a_1,\hdots,a_n$.
Thus, comparing \eqref{Slimit-rel.d} and \eqref{eq:EigenvalueModulusDensity}
and noting that
\begin{align}
  \text{perm}\Big(a_j^{2(k-1)}\Big)_{j,k=1,\hdots,n}
= \text{perm}\Big(a_j^{2(n-k)}\Big)_{j,k=1,\hdots,n}
= \sum_{\sigma \in \mathbb{S}_n} \prod_{j=1}^{n} a_j^{2(n-\sigma(j))}
\end{align}
is the symmetrization of $\prod_{j=1}^{n} a_j^{2(n-j)}$, we find that 
the joint distribution of the radii of the eigenvalues is the symmetrization of that
of the vector $(R_{11},\hdots,R_{nn})$ in Lemma~\ref{lemma:R-MellinTransform}.

Thus, it remains to obtain the limiting distribution of the vector $(R_{11},\hdots,R_{nn})$.
Here we may use that, by multivariate Mellin inversion,
\begin{align}
f_{(R_{11},\hdots,R_{nn})}(x_1,\hdots,x_n)
= \int_{\myreal^n} \mellin_{(R_{11}^2,\hdots,R_{nn}^2)}(\pmb{1}_n + \imath s) \, \prod_{j=1}^{n} x_j^{-(1 + \imath s_j)} \, \frac{ds_j}{2\pi} \,,
\end{align}
where the first factor inside the integral is equal to $\mys(\sigma'(\varrho') + \imath s)$ by Lemma \ref{lemma:R-MellinTransform}. Thus, the expansion of the Mellin transform at the point $\pmb{1}_n$ is the same as that of the spherical transform at the point $\sigma'(\varrho')$. After replacing $x$ with $\widetilde{x}^M$ and repeating the argument leading to Eq.~\eqref{eq:result-7}, we arrive at the second part of Prop.~\ref{prop:CLT}.

Finally, for the special case of a polynomial ensemble of derivative type,
the matrix $\Sigma$ is again diagonal, which obviously yields the second part of Prop.~\ref{prop:CLT-DPE}. 

\pagebreak[2]

The circumstance that the spherical transform of a polynomial ensemble of derivative type factorizes also has consequences in the non-asymptotic context. In particular, we can use it to characterize the poly\-nomial ensembles of derivative type:

\begin{corollary}[Characterization of Polynomial Ensembles of Derivative Type]
\label{cor:Characterization}
A~bi-unitarily invariant random~matrix $X$ on $\GL(n,\mycmplx)$ with a density $f \in L^{1,K}(G)$
is from a poly\-nomial ensemble of derivative type if and only if 
the diagonal elements $R_{11},\hdots,R_{nn}$ of the~upper triangular matrix $R$ 
in the corresponding QR-decompo\-sition $X = QR$ are independent
and $R_{nn}^2$ has a density in $L^{1,n-1}_{[1,n]}(\myreal_{+})$.
\end{corollary}

\begin{proof}
Let $X$ be a random matrix from the polynomial ensemble of derivative type associated with the weight function $\omega\in L_{[1,n]}^{1,n-1}(\mathbb{R}_+)$. Then the spherical transform $\mys_X$ factorizes by Corollary~\ref{cor:SphericalTransform}. By Lemma~\ref{lemma:R-MellinTransform}, this implies 
that the Mellin trans\-form of the vector $(R_{11}^2,\hdots,R_{nn}^2)$ factorizes,
which means that the diagonal elements of $R$ are independent random variables. Moreover, the densities of the random variables $R_{kk}^2$ are given by the family $\{x^{n-k}\omega(x)/\mellin\omega(n\!-\!k\!+\!1)\}_{k=1,\ldots,n}$. In particular, that of $R_{nn}^2$ is in $L^{1,n-1}_{[1,n]}(\myreal_{+})$ because $\omega$ is.

Conversely, suppose that the diagonal elements of $R$ are independent random variables. Let $h_1,\hdots,h_n$ denote the densities of their squares. Then
\begin{equation}
\mellin_{(R_{11}^2,\hdots,R_{nn}^2)}(s) = \prod_{j=1}^{n} \mellin h_j (s_j).
\end{equation}
By Lemma~\ref{lemma:R-MellinTransform}, it follows that for any $s \in \operatorname{conv}(\mathbb{S}(\varrho')) + \imath \mathbb{R}^n$ we have
\begin{equation}
\mys_X(s) = \prod_{j=1}^{n} \mellin h_j (s_j -\varrho'_{n-j+1}+1) .
\end{equation}
The product of these functions must be symmetric in $s = (s_1,\hdots,s_n)$ by the basic properties of the spherical transform. But a product of functions which are not identically zero is symmetric only if all the functions are proportional to one another. We therefore obtain
\begin{equation}
\mellin h_j(s_j-\varrho'_{n-j+1}+1)=\frac{\mellin h_n(s_j-\varrho'_{1}+1)}{\mellin h_n(\varrho'_{n-j+1}-\varrho'_{1}+1)}=\frac{\mellin h_n(s_j-(n-1)/2)}{\mellin h_n(n-j+1)}
\end{equation}
for all $s_j \in [(n+1)/2,(3n-1)/2]$ and $j=1,\ldots,n$, because $\mellin h_j(1) = 1$. Setting $\omega := h_n\in L^{1,n-1}_{[1,n]}(\myreal_{+})$, the Mellin-transform $\mellin \omega$ is defined on the set $[1,n] + \imath \mathbb{R}$ and
\begin{equation}
\mys_X(s) = \prod_{j=1}^{n} \frac{\mellin \omega(s_j-(n-1)/2)}{\mellin \omega(n-j+1)}.
\end{equation}
Hence, in view of the uniqueness theorem for the spherical transform and Corollary \ref{cor:SphericalTransform}, we may draw the conclusion that $X$ belongs to the polynomial ensemble of derivative type associated with the function $\omega$.
\end{proof}

As a byproduct of the proof we have found a simple relation between the densities of the diagonal elements $R_{11},\hdots,R_{nn}$ in the QR-decomposition and the density $\omega$. The density of $R_{jj}^2$ is given by $h_j(r)=r^{n-j}\omega(r)/\mellin\omega(n-j+1)$, $j=1,\ldots,n$.

\pagebreak[2]

\section{Conclusion}
\label{sec:conclusion}

With the help of the spherical transform we have investigated the distributions of products of independent  bi-unitarily invariant random matrices. In particular, we have derived transformation formulas for the joint probability densities of the squared singular values of a polynomial ensemble as well as the associated kernels and bi-orthogonal functions. These transformation formulas refer to the situation where a random matrix from the corresponding random matrix ensemble is multiplied by an independent random matrix from a polynomial ensemble of derivative type. These results generalize existing counterparts for the case that the polynomial ensemble of derivative type is an induced Laguerre (Ginibre / Gaussian) ensemble or an induced Jacobi (trun\-cated unitary) ensemble, see Ref.~\cite{KS:2014,Kuijlaars:2015,CKW:2015}.

Furthermore, for the densities on the matrix space $G={\rm GL}(n,\mathbb{C})$ arising from the independent products of $p$ Ginibre matrices and $q$ inverse Ginibre matrices, we were able to construct an analytic continuation in $p$ and $q$. However, the resulting density is a probability density only under certain restrictions on the parameters $p$ and $q$, see Eq.~\eqref{eq:NiceParameters}. Otherwise, the resulting matrix density and, hence, the induced joint density of the singular values are signed, which is very similar to Gindikin's theorem \cite{FK} about interpolations of sums of Wishart matrices. In particular our observation implies that the Ginibre ensemble is not infinitely divisible with respect to multiplication. In contrast to the behavior of the singular values, the induced joint density of the complex eigenvalues is always positive, for any positive values of $p$~and~$q$, and even a probability density. This intriguing behavior underlines our claim in Ref.~\cite{KK:2016a} that not every isotropic eigenvalue density corresponds to a bi-unitarily invariant random matrix ensemble. The question how rich the subset of those eigenvalue densities corresponding to a bi-unitarily invariant random matrix ensemble really is remains open.

As a third example of our approach we have shown how the results in Ref.~\cite{Reddy:2016} about the Lyapunov exponents (logarithms of the singular values) are related to the spherical transform. In particular, we have sketched an alternative way to show that the Lyapunov exponents are asymptotically Gaussian distributed when the number of factors in a product of independently and identically distributed, bi-unitarily in\-variant random matrices becomes large. In other words, the singular values are asymptotically log-normally distributed. This derivation is reminiscent of the standard proof of the central limit theorem for random vectors, 
and might open a way to extend the central limit theorem for Lyapunov exponents to the situation where the limiting distributions are not Gaussian but heavy-tailed.
In addition to that, our discussion has revealed that the Lyapunov exponents are asymptotically independent when the underlying random matrices belong to a polynomial ensemble of derivative type.


\section*{Acknowledgements}

We thank Gernot Akemann and Friedrich G\"otze for fruitful discussions. More\-over we acknowledge financial support by the {\it CRC 701: ``Spectral Structures and Topological Methods in Mathematics''} of the {\it  Deutsche Forschungsgemeinschaft}.

\bigskip

\end{document}